\documentclass[letterpaper,10pt,reqno]{amsart}
\usepackage{amsmath,amsfonts,amssymb,amsthm,mathabx}
\usepackage[utf8]{inputenc} 
\usepackage{xargs}
\usepackage{color}
\usepackage[colorlinks=true,citecolor=forestgreen]{hyperref}
\usepackage{bm}
\usepackage{paralist} 
\usepackage{ifthen}
\usepackage{tikz}
\usepackage{tikz-cd} 
\usetikzlibrary{calc,decorations.pathreplacing,decorations.markings,arrows,decorations.pathmorphing}
\makeatletter
\newcommand{\gettikzxy}[3]{%
	\tikz@scan@one@point\pgfutil@firstofone#1\relax
	\edef#2{\the\pgf@x}%
	\edef#3{\the\pgf@y}%
}
\makeatother

\calclayout

\theoremstyle{plain}
\newtheorem{theorem}{Theorem}[section]
\newtheorem{corollary}[theorem]{Corollary}
\newtheorem{proposition}[theorem]{Proposition}
\newtheorem{algorithm}[theorem]{Algorithm}

\theoremstyle{definition}
\newtheorem{definition}[theorem]{Definition}
\newtheorem{remark}[theorem]{Remark}
\newtheorem{example}[theorem]{Example}
\newtheorem{question}[theorem]{Question}

\definecolor{darkblue}{rgb}{0,0,0.7} 
\definecolor{forestgreen}{rgb}{0.07,0.35,0.10} 
\newcommand{\darkblue}{\color{darkblue}} 
\newcommand{\defn}[1]{\emph{\darkblue #1}} 

\newcommandx{\pol}[1][1=P]{\mathsf{#1}} 

\newcommand{\np}[1]{\langle#1\rangle}  
\newcommand{\R}{\mathbb{R}}

\DeclareMathOperator{\aff}{aff} 
\DeclareMathOperator{\conv}{conv} 
\DeclareMathOperator{\cone}{cone} 
\DeclareMathOperator{\ncone}{ncone} 

\newcommand\bA{\bm{A}}
\newcommand\be{\bm{e}}
\newcommand\bi{\bm{i}}
\newcommand\bj{\bm{j}}
\newcommand\bv{\bm{v}}
\newcommand\bw{\bm{w}}
\newcommand\by{\bm{y}}
\newcommand\bs{\bm{s}}
\newcommand\bx{\bm{x}}

\title{Lineup polytopes of product of simplices}

\author[F.~Castillo]{Federico Castillo}
\address[F.~Castillo]{Pontificia Universidad Cat\'olica de Chile, Avenida Vicu\~na Mackenna 4860, Macul, Chile}
\email{efecastillo.math@gmail.com}
\urladdr{\url{https://sites.google.com/view/fcastillo}}

\author[J.-P.~Labbé]{Jean-Philippe Labbé}
\address[J.-P. Labb\'e]{École de Technologie Supérieure, 1111 rue Notre-Dame Ouest, Montréal (Qc) H3C 6M8, Canada}
\email{jean-philippe.labbe@etsmtl.ca}
\urladdr{\url{https://jplab.github.io/}}


\begin{document}

\begin{abstract}
Consider a real point configuration $\bA$ of size $n$ and an integer $r \leq n$.
The vertices of the $r$-lineup polytope of $\bA$ correspond to the possible orderings of the top $r$ points of the configuration obtained by maximizing a linear functional. 
The motivation behind the study of lineup polytopes comes from the representability problem in quantum chemistry.
In that context, the relevant point configurations are the vertices of hypersimplices and the integer points contained in an inflated regular simplex.
The central problem consists in providing an inequality representation of lineup polytopes as efficiently as possible.
In this article, we adapt the developed techniques to the quantum information theory setup.
The appropriate point configurations become the vertices of products of simplices.
A particular case is that of lineup polytopes of cubes, which form a type $B$ analog of hypersimplices, where the symmetric group of type~$A$ naturally acts.
To obtain the inequalities, we center our attention on the combinatorics and the symmetry of products of simplices to obtain an algorithmic solution.
Along the way, we establish relationships between lineup polytopes of products of simplices with the Gale order, standard Young tableaux, and the Resonance arrangement.
\end{abstract}

\maketitle

\section*{Introduction}

The Farkas--Minkowski--Weyl Theorem establishes a duality principle which is fundamental in discrete geometry: convex polyhedra admit two different equivalent representations \cite[Theorem 7.1]{schrijver_theory_1998}.
Either they represent the set of solutions of a system of linear inequalities ($H$-representation) or they are sums of a linear subspace, a pointed cone, and a polytope ($V$-representation).
Certain problems---for example, asking whether a point belongs to a convex polyhedra, the so-called \emph{membership problem}---are easily solvable if one has access to its $H$-representation, but not if one only has its $V$-representation.
The reverse direction is similarly true, making the translation between representations a task of major importance which is well-known to be computationally expensive \cite{avis_good_1995}.
This problem is sometimes refered to as the \emph{representation conversion problem} or the \emph{convex hull problem}.
The existence of a polynomial time translation algorithm appears to be unlikely as it is NP-complete for unbounded polyhedra \cite{khachiyan_generating_2008}.
It remains an open problem to determine whether there is a translation algorithm that runs polynomially (on the size of input and output) for bounded polyhedra \cite[Open Problem 26.3.4]{toth_handbook_2017}.
By exploiting symmetry of polyhedra, it is possible to obtain more efficient algorithms using the related geometric and combinatorial objects such as fundamental domains and posets.
Indeed, in the present paper, we adapt a $(V\to H)$-translation algorithm introduced in \cite{castillo_effective_2021} and extend its use to another family of symmetric polytopes: lineup polytopes of product of simplices.
Aside from their geometric origin, it turns out that lineup polytopes of product of simplices show relations to quantum information theory \cite{klyachko_quantum_2006}, to the White Whale \cite{deza_sizing_2022}, and to applications of standard Young tableaux in deconvolution in mathematical statistics \cite{mallows_deconvolution_2007,mallows_young_2015}.

\textbf{Quantum marginal problem.} The motivation for extending this algorithm to product of simplices comes from quantum information theory.
Almost 20 years ago, Klyachko used tools from representation theory to study the quantum marginal problem (QMP), details of which are provided in his unpublished manuscript \cite{klyachko_quantum_2004}.
His main contribution is an $H$-representation of the (moment) polytope of all compatible marginals.
Each inequality in the representation has physical significance: it gives a linear constraint on the allowable marginals that are simple to test in practice.
The more general problem of providing an $H$-representation of moment polytopes has been treated by Berenstein--Sjaamar \cite{berenstein_coadjoint_2000}, Ressayre \cite{ressayre_geometric_2010}, and Vergne--Walter \cite{vergne_inequalities_2017}. 
All of these $H$-representations are hard to make effective in practice.
In the article \cite{castillo_effective_2021}, we lay out discrete geometric and combinatorial methods in order to circumvent the complexity of Klyachko's framework by relaxing the problem and computing a larger polytope \emph{while keeping the physically relevant portion} of Klyachko's solution.
The main geometric tool introduced therein are \emph{lineup polytopes}, whose $H$-representations provide necessary linear inequalities that we can effectively compute.

\textbf{Parallel computational tools and symmetry}
There are several translating algorithms between the $V$- and $H$-representations that are already implemented, see \cite{avis_pivoting_1991},  \cite{clarkson_algorithms_1998}, \cite{chand_algorithm_1970}, and  \cite{rote_degenerate_1992} for some examples.
Each algorithm seems to do well on certain classes of polytopes, but none always stand out.
In the present study, we examine a particular family and tailor our methods to that context.
We start with a known normal fan and the goal is to compute a specific refinement of it.
We define lineup fans to serve as intermediate steps, obtained by successive refinements.
The refinement is obtained by adding certain hyperplanes to each full-dimensional cones of the intermediate fans.
This idea is similar to the one behind the \emph{incremental} algorithms which compute convex hulls by adding one hyperplane at the time.
The refinement of a fan naturally lends itself to parallelization; for a recent study on paralellization of general $(V\to H)$-algorithms see  \cite{avis_mplrs_2018}.
Another feature allowing us to speed the computations is the presence of symmetry.
The family of polytopes at play is highly symmetric and we exploit this fact to compute orbit representatives instead of all of them.
Finally, the combinatorics of the problem at hand provide a poset leading to the refinements needed at each step. 

\textbf{White Whale}
Lineup polytopes of hypercubes (i.e. products of line segments) are related to the \emph{resonance arrangement}, see Remark \ref{rem:white_whale}.
This arrangement has the universal property that any rational hyperplane arrangement is the minor of some large enough resonance arrangement \cite{kuhne_universality_2020}.
The corresponding zonotope, known as the White Whale, is the the Minkowski sum of all 0/1 vectors of length $ N $.
There has been interest in computing the number of vertices but even with the latest available method the problem remains ellusive for $N>9$ \cite{deza_sizing_2022}.

\textbf{Realizable Tableaux}
Another case appeared in disguise in earlier work.
The lineup polytope of the product of two simplices $ \Delta_{d-1} \times \Delta_{e-1} $ is related to the number of \emph{realizable} standard Young tableaux (SYT) of rectangular shape $ d \times  e $, as observed by Klyachko in \cite{klyachko_quantum_2004}.
Realizable SYT are also called \emph{outer sums} and they are systematically studied by Mallows and Vanderbei \cite{mallows_young_2015}.
They appear also in recent work of Black and Sanyal \cite{black_flag_2022}.
Contrary to the set of all SYT, which has a close product formula (the hook length formula \cite{greene_probabilistic_1982}) there is no enumeration formula for the realizable case.
Recently Araujo, Black, Burcroff, Gao, Krueger, and McDonough provide some asymptotic results for realiable SYT of rectangular shape in \cite{araujo_realizable_2023} .
Therein, they prove that with $d$ fixed, the number of such tableaux is exponential in $e$ but the base of the exponential is still unknown.
Our computations shed a light on what that base may be.

\textbf{Organization of the paper} 
In Section~\ref{sec:prelim}, we provide the preliminaries on polytopes, their normal fans, lineup polytopes, the connection to the physical motivation and describe important examples.
In Section~\ref{sec:gen_prod}, we develop general results for lineups of product of simplices along with the algorithmic method.
In Section~\ref{sec:hypercubes}, we specialize our tools to the particular case of product of line-segments.
In Section~\ref{sec:further}, we finish with some observations about the original quantum marginal problem and lineup polytopes of cyclic polytopes.

\textbf{Acknowledgements}
The authors thank Alex Black, Jesus De Loera, Julia Liebert, Arnau Padrol, and Christian Schilling for helpful conversations.
In particular, we are thankful to Vic Reiner for pointing out the connection with the poset in Remark~\ref{rem:stan}.
We are grateful to the Simons Center for Geometry and Physics where a part of the work was carried out and to Christian Schilling's group at the Ludwig Maximilian Universit\"at M\"unchen for their hospitality, and providing a fruitful research atmosphere.
FC was partially supported by FONDECYT Grant 1221133.

\section{Preliminaries}
\label{sec:prelim}
We adopt the following conventions: ${d\in\mathbb{N}\setminus\!\{0\}}$, $[d]:=\{1,2,\dots,d\}$.
The cardinality of a set $S$ is denoted by~$|S|$.
Let $\mathbb{R}^d$ be the $d$-dimensional Euclidean space with elementary basis $\{\be_i~:~i\in[d]\}$ and inner product given by $\np{\be_i,\be_j}=\delta_{i,j}$ for $i,j\in[d]$.
Whenever a vector $\mathbf{v}\in\mathbb{R}^d$ is written as a tuple $(v_1,\ldots,v_d)$, the entries are expressing the coefficients of $\mathbf{v}$ in the standard basis, i.e., $\mathbf{v}=\sum v_i \mathbf{e}_i$. 

\subsection{Polytopes and normal fans}

A \defn{polyhedron} is the intersection of finitely many closed halfspaces \cite[Chapter~7]{schrijver_theory_1998}: 
\begin{equation}\label{eq:H-rep}
	\pol[Q] := \left\{ \mathbf{x}\in\mathbb{R}^d: M\mathbf{x} \leq \mathbf{b} \right\},
\end{equation}
where $M$ is a matrix and $\mathbf{b}$ is a vector.
The expression in \eqref{eq:H-rep} is a \defn{$H$-representation} of $\pol[Q]$.
A row of $M$ and its corresponding entry in $\mathbf{b}$ gives a \defn{defining inequality} of $\pol[Q]$, and represents a closed halfspace containing~$\pol[Q]$.
If a row of $M$ is a positive linear combination of other rows of $M$, it is not necessary to define $\pol[Q]$, and this $H$-representation is called \defn{redundant}.

Let $\mathbf{A}=\{\mathbf{v}_1,\dots,\mathbf{v}_n\}\subset\mathbb{R}^d$.
We refer to $\mathbf{A}$ as a \defn{point configuration of size} $n$.
The \defn{affine hull} $\aff(\mathbf{A})$ of~$\mathbf{A}$ is the set of vectors $\sum_{i=1}^n\lambda_i\mathbf{v}_i$ such that $\sum_{i=1}^{n}\lambda_i=1$ and $\lambda_i\in\mathbb{R}^d$.
The \defn{conical} (or ``positive'') \defn{hull} $\cone(\mathbf{A})$ of~$\mathbf{A}$ is the set of vectors $\sum_{i=1}^n \lambda_i\mathbf{v}_i$ such that $\lambda_i\geq 0$, defining a \defn{cone}.
A cone is \defn{pointed} if it contains no lines.
The \defn{convex hull} $\conv(\mathbf{A})$ of~$\mathbf{A}$ is the set of vectors $\sum_{i=1}^n \lambda_i\mathbf{v}_i$ such that ${\sum_{i=1}^{n}\lambda_i=1}$ and $\lambda_i\geq 0$, defining a \defn{polytope}.
The elements in these sets are called \defn{affine}, \defn{conical} and \defn{convex combinations} of~$\mathbf{A}$, respectively.
We refer to the elements of minimal generating sets of affine, conical and convex hulls as \defn{line generators}, \defn{ray generators}, and \defn{vertices}.
Line-generators are unique up to change of affine basis and ray generators are unique up to scaling by positive scalars.

By the Farkas--Minkowski--Weyl theorem, every polyhedron $\pol[Q]$ can be decomposed uniquely as the sum of an affine hull, a conical hull and a convex hull:
\begin{equation}\label{eq:structure}
\pol[Q]=\pol[L]+\pol[K]+\pol,
\end{equation}
where $\pol[L]$ is a linear subspace (called the \defn{lineality space} of $\pol[Q]$),~$\pol[K]$ is a pointed cone (called the \defn{recession cone} of $\pol[Q]$), and $\pol[P]$ is a polytope.
The expression in \eqref{eq:structure} is \defn{$V$-representation} of~$\pol[Q]$.
Thus, polytopes and cones are polyhedra: polytopes are \emph{bounded} polyhedra and cones are \emph{homogeneous} polyhedra that is, $\mathbf{b}=\mathbf{0}$ in \eqref{eq:H-rep}. 

\begin{remark}
Translating between $V$- and $H$-representation of affine or linear subspaces is done quite efficiently through Gauss elimination.
For polytopes or pointed cones, this process is known to be much harder and a central subject in linear optimization and discrete geometry as mentioned in the introduction.
\end{remark}

\subsubsection{Adopted technique: exploiting duality}
\label{ssec:adopted}

We restrict ourselves to the case of polytopes with a high level of symmetry.
In order to pass from a $V$- to a $H$-representation, we use the following method.
Effectively, it turns a $(V\to H)$-translation into a $(H\to V)$-translation which is easier to handle in the special cases of interest.

Let $\pol=\conv(\mathbf{A})$ be a polytope. 
A linear inequality satisfied by all point $\mathbf{x} \in \pol$ is called \defn{valid}.
The \defn{support function} $h_{\pol}:\mathbb{R}^d\to \mathbb{R}$ of $\pol$ is defined as $h_{\pol}(\mathbf{y}):=\max_{\mathbf{x}\in \pol} \np{\mathbf{y},\mathbf{x}}$.
Every vector $\mathbf{y}\in\mathbb{R}^d$ induces a unique valid inequality on a polytope~$\pol$, according to $\np{\mathbf{y},\mathbf{x}}\leq h_{\pol}(\mathbf{y})$.
The polytope $\pol^\mathbf{y} = \{\mathbf{x}\in \pol~:~\np{\mathbf{y},\mathbf{x}}=h_{\pol}(\mathbf{y}) \}$ is refered to as a \defn{face} of~$\pol$.
Vertices of $\pol$ are 0-dimensional faces and \defn{facets} of $\pol$ are codimension-1 faces.
Given a face $\pol[F]$ of $\pol$, we define its \defn{open} and \defn{closed normal cones}:
\begin{align}
\begin{array}{l@{\vspace{1mm}}l}
\ncone_{\pol}(\pol[F])^\circ &:= \{\mathbf{y}\in\mathbb{R}^d~:~\pol^\mathbf{y}=\pol[F]\} \text{ and} \\
\ncone_{\pol}(\pol[F])       &:= \{\mathbf{y}\in\mathbb{R}^d~:~\pol^\mathbf{y}\supseteq \pol[F]\}. \label{eq:ncone_closed}
\end{array}
\end{align}
The collection $\mathcal{N}(\pol):=\{\ncone_{\pol}(\pol[F])~:~\pol[F]\text{ a face of }\pol\}$ is the \defn{normal fan} of $\pol$. 
Here is the keystone of the approach: the normal fan of a polytope is \emph{entirely recovered} from the normal cones of the vertices, since their faces \emph{are} all other cones as the following example illustrates.

\begin{example}[Normal fan of a polygon on the plane]
Let $\pol=\conv\left\{(0,0),(3,0),(3,1),(1,2),(0,2)\right\}$ in $\mathbb{R}^2$ as illustrated in Figure~\ref{fig:normal_fan}.

\begin{figure}[!htbp]
\includegraphics[scale=0.2]{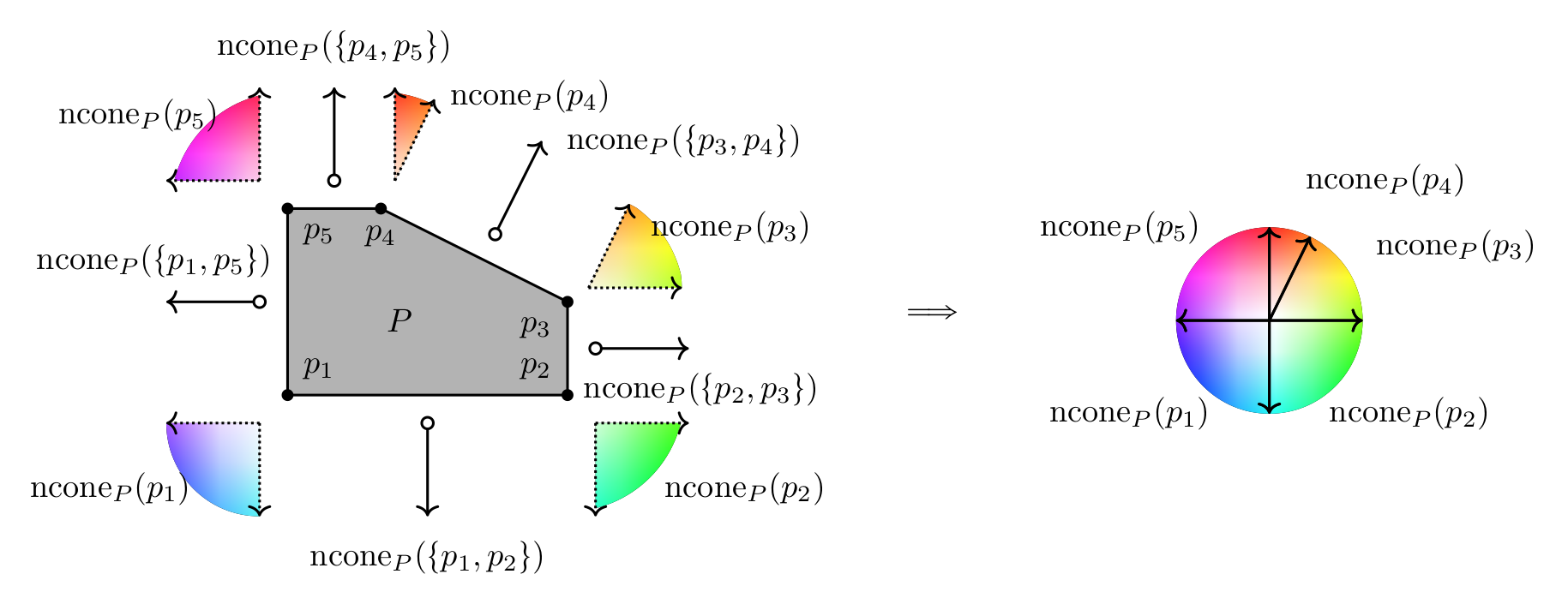}
\caption{A $2$-dimensional polytope and its normal fan obtained by placing the normal cones of the vertices of $\pol$ at the origin.}
\label{fig:normal_fan}
\end{figure}

\end{example}

\noindent
To go from a $V$- to an $H$-representation, one first determines the normal cones of the vertices, then obtain all rays, and finally get a non-redundant $H$-representation:

\begin{compactenum}[i.]
\item The definitions in Equations~\eqref{eq:ncone_closed} say that the normal cone of a face $\pol[F]$ consists of all vectors $\mathbf{y}$ whose linear functional is maximized on $\pol[F]$. 
Whence, for each vertex $\mathbf{v}$ of $\pol$ its normal cone has the following $H$-representation.
\begin{equation}\label{eq:ncone_vertex}
\ncone_{\pol}(\mathbf{v}) = \{\mathbf{y}\in\mathbb{R}^y : \np{\mathbf{y},\mathbf{v}-\mathbf{v}'}\geq0, \text{ for }\mathbf{v}'\in \mathbf{A}\}.
\end{equation}
\item Decompose the normal cone into its lineality space and its recession cone as in \eqref{eq:structure}, i.e., $\ncone_{\pol}(\mathbf{v})=\pol[L]_{\pol}(\mathbf{v})+\pol[K]_{\pol}(\mathbf{v})$ (there is no polytope factor in this case).
The lineality space $\pol[L]_{\pol}(\mathbf{v})$ is the orthogonal complement of $\aff(\pol)$, hence it does not depend on $\mathbf{v}$.
\item Translate the $H$-representation of $\pol[K]_{\pol}(\mathbf{v})$ to a $V$-representation: 
\begin{equation*}
\pol[K]_{\pol}(\mathbf{v}) = \left\{\sum_{i=1}^j \lambda_i\mathbf{r}_i~:~ \lambda_i\geq0, \text{ for }i\in[j]\right\},
\end{equation*}
for some $\mathbf{r}_1,\dots,\mathbf{r}_j\in\mathbb{R}^d$.
\item Let $\mathbf{B}$ be the set of all ray generators $\mathbf{r}_i$'s found in the previous step for all $\ncone_{\pol}(\mathbf{v})$.
For each $\mathbf{r}\in \mathbf{B}$, we determine the value of $h_{\pol}(\mathbf{r})$ by evaluating it on $\mathbf{A}$.
\end{compactenum}
We end up with the non-redundant $H$-representation
\[
\pol=\left\{\mathbf{x}\in\mathbb{R}^d~:~\begin{array}{c}\np{\mathbf{r},\mathbf{x}}\leq h_{\pol}(\mathbf{r}),\text{ for all }\mathbf{r}\in \mathbf{B}\\ \np{\mathbf{y},\mathbf{x}}= h_{\pol}(\mathbf{y}) ,\text{ for all }\mathbf{y}\in \pol[L]_{\pol}\end{array}\right\}.
\]

The high level of symmetry of the studied polytopes confers optimal efficiency to this approach.
Indeed, the exponential number of vertices to consider is reduced to a minimum.
Furthermore, the description in Equation~\eqref{eq:ncone_vertex} is reduced to treating only one linear functional and vertex per orbit.
The main remaining piece is the $(H\to V)$-translation in the third step.

\subsection{Lineup polytopes}
Let $\mathbf{A}=\{\mathbf{v}_1,\dots,\mathbf{v}_n\}\subset\mathbb{R}^d$ and $\by\in\R^d$.
A generic vector~$\by$ provides an injective linear functional $\np{\by,\cdot}:\R^d\to\R$ that totally orders the elements of $\mathbf{A}$ from maximal to minimal value. 
Given such a total order and an integer $r$ such that $1\leq r\leq n$, the sequence $\ell$ of the first $r$ elements in this total order is called a \defn{lineup of length~$r$} of $\mathbf{A}$ \cite[Definition~6.1]{castillo_effective_2021}.
Let $\bw:=(w_1,w_2,\dots,w_r)$ be such that $1> w_1 > w_2 > \cdots > w_r > 0$ and $\sum_{i=1}^r w_i=1$ and furthermore let $\ell=(\bv_1,\bv_2,\dots,\bv_r)$ be an ordered list of~$r$ vectors of $\mathbf{A}$.
We refer to the $w_i$'s as \defn{weights}.
The \defn{occupation vector} associated to $\ell$ with respect to $\bw$ is $\mathbf{o}_{\bw}(\ell):= \sum_{i=1}^{r}w_i\bv_i$.
The \defn{$r$-lineup polytope} of $\mathbf{A}$ is 
\[
\pol[L]_{r,\bw}(\mathbf{A}):= \conv\left\{\mathbf{o}_{\bw}(\ell)~:~\ell \text{ is an ordered $r$-subset of }\mathbf{A}\right\},
\]
see \cite[Definition 6.1]{castillo_effective_2021}.
To a polytope $\pol\subset\R^d$, we associate the point configuration $\mathbf{A}$ given by listing its vertices in some order.
The next proposition summarizes the content of \cite[Theorem E]{castillo_effective_2021}; part (3) is a slight generalization of \cite[Proposition 6.18]{castillo_effective_2021}.

\begin{proposition}\label{prop:main_structure}
Let $\mathbf{A}=\{\mathbf{v}_1,\dots,\mathbf{v}_n\}\subset\mathbb{R}^d$ and $1\leq r\leq n$.
If $\bw\in\R^r$ has strictly decreasing coordinates, then the following are statements hold:
\begin{enumerate}
\item The point $\mathbf{o}_{\bw}(\ell)$ is a vertex of the lineup polytope $\pol[L]_{r,\bw}(\mathbf{A})$ if and only if $\ell$ is an $r$-lineup.
\item If $\bv^\ell$ is a vertex of $\pol[L]_{r,\bw}$ with lineup $\ell=(\bv_1,\dots,\bv_r)$, then the open normal cone of $\bv^\ell$ is the set of $\by\in\R^d$ such that
\begin{align*}
& \np{\by,\bv_1}>\np{\by,\bv_2}>\cdots>\np{\by,\bv_r} \text{ and}\\
& \np{\by,\bv_r}>\np{\by,\bv}, \text{	for every point $\bv$ of $\mathbf{A}$ not contained in $\ell$.}
\end{align*}
The normal cone of $\bv^\ell$ is called the \defn{lineup cone} of $\ell$.
\item If $\bs(\by,\mathbf{A})$ denotes the largest $r$ values of $\np{\by,\mathbf{A}}$ ordered decreasingly, then
\begin{equation*}
\max_{\bx\in\pol[L]_{r,\bw}(\mathbf{A})} \,\np{\by,\bx}= \np{\bw,\bs(\by,\mathbf{A})}.
\end{equation*}
\end{enumerate}
\end{proposition}

Lineup cones are independent of the specific values of the entries of $\bw$ as long as they are strictly decreasing.
Therefore, the normal fan $\Sigma^r(\mathbf{A})$ is independent of the specific choice of $\bw$ and we call it the \defn{lineup fan} of~$\mathbf{A}$.
For this reason, we omit the symbol $\bw$ in our notation.
Also when $r=n$ we omit the symbol~$r$ and call $\pol[L]_{n,\bw}(\mathbf{A})=\pol[L](\mathbf{A})$ the \defn{sweep polytope} of $\mathbf{A}$.
These polytopes were studied by Padrol and Philippe, see \cite{padrol_sweeps_2021}.
Sweep polytopes are zonotopes: 
\begin{equation*}
\pol[L](\mathbf{A}) = \sum_{\bv_1,\bv_2\in\mathbf{A}} \text{conv}(\bv_{1}, \bv_{2}).
\end{equation*}
Frequently, the $H$-representations of these zonotopes are difficult to obtain.
However, the general definition with $r<|\mathbf{A}|=n$ allows us to partially compute the normal fan of the sweep polytope using recursion.
Indeed, we have
\begin{equation}
	\label{eq:refinement}
	\mathcal{N}(\text{conv}(\mathbf{A}))=\Sigma^1(\mathbf{A})\succeq\Sigma^2(\mathbf{A})\succeq\dots\succeq\Sigma^n(\mathbf{A})=\mathcal{N}(\pol[L](\mathbf{A})),
\end{equation}
where $\Sigma\succeq\Sigma'$ denotes that $\Sigma'$ is a refinement of $\Sigma$.

\begin{remark}\label{rem:partial}
Equation \eqref{eq:refinement} delivers partial information of the whole sweep polytope without fully computing it.
This feature is interesting on its own right as it is possible to obtain relevant information (a non-redundant and partial facet-defining $H$-representation) for a potentially large polytope without having to wait for the translation algorithm to complete. 
\end{remark}

\subsection{Connections with physics}
\label{ssec:conn_phys}
We describe the physical context related to the study of lineup polytopes of products of simplices.
We start with the simplest instance of QMP (for a general account of QMP see \cite{schilling_quantum_2015}).
Let $\mathcal{H}_{1},\mathcal{H}_{2}$ be two finite dimensional Hilbert spaces. 
There is a linear map called the \defn{partial trace} between $ \text{End}\left(  \mathcal{H}_{1}\otimes\mathcal{H}_{2} \right) $ and $ \text{End}\left(  \mathcal{H}_{1} \right)$ uniquely determined by mapping $ \rho_{1}\otimes\rho_{2} \mapsto \text{trace}(\rho_{2})\rho_{1} $  whenever $\rho_1,\rho_2$ are linear operators on $ \mathcal{H}_{1} $ and $ \mathcal{H}_{2} $ respectively.
Using the partial trace, we associate to an operator $ \rho \in   \text{End}\left(  \mathcal{H}_{1}\otimes\mathcal{H}_{2} \right) $ its \defn{marginals} $ \rho^{(1)} $ and $ \rho^{(2)} $.
Furthermore, if we assume that $ \rho $ is a density operator, (that is, if all its eigenvalues are real, in the interval $ [0,1] $, and they add up to 1), then its two marginals $ \rho^{(i)} $ are density operators too.
In this context, the quantum marginal problem is to determine the triples $ (\bm{w}, \bm{w}^{(1)}, \bm{w}^{(2)}) $ such that there exists an operator $ \rho $ such that $ \bm{w} = \text{spec}(\rho) $ and $ \bm{w}^{(i)} = \text{spec}(\rho^{(i)}),$ for $i = 1,2$.

The setup naturally generalizes to the tensor product of $ N \geq 2 $ Hilbert spaces each of dimension $ d $: we seek to relate the spectrum of a density operator on the tensor space with the spectra of its $ N $ marginals.
This space parametrizes the states of a system of $N$ distinguishable particles, called qu\emph{d}its if we want to refer to the dimension $d$.
We shall give special attention to the case $ d=2 $ because it corresponds to systems of qubits relevant in quantum information.
This instance of QMP was solved by Klyachko and Altunbulak, see \cite{klyachko_quantum_2006} and \cite{altunbulak_pauli_2008}.
For any vector $ \bm{x} \in \mathbb{R}^{d} $ we define $ \bm{x}^{\downarrow} $ as the vector consisting of the absolute values of the entries of $ \bm{x} $ in weakly decreasing order.
The set
\begin{equation}\label{eq:klypo}
\Lambda(d,N,\bm{w}):= \left\{ \left(\bm{v}^{1},\dots,\bm{v}^{N} \right) \in \prod_{i=1}^N \mathbb{R}^{d} ~\middle|~ \exists\ \rho \text{ with } \text{spec}(\rho)=\bm{w} \ \text{ and }\ \text{spec}\left(\rho^{(i)}\right)^{\downarrow}=\bm{v}^{i}, \text{ for all } i \in [N] \right\} 
\end{equation}
\noindent
of spectra of the marginals arising from operators with a fixed spectrum $ \bm{w}^{\downarrow}  $ is a polytope.
Polytopality is a consequence of general results about moment polytopes \cite{kirwan_convexity_1984}.
Klyachko described a finite set of defining inequalities in \cite[Theorem 4.2.1]{klyachko_quantum_2006} and \cite[Example 2]{altunbulak_pauli_2008}.
His solution has two steps:
\begin{enumerate}
\item (Discrete Geometry) Find all \textit{edges} of certain polyhedral cones called \textit{cubicles}.
\item (Schubert Calculus) For each \textit{edge} consider all permutations satisfying certain cohomological condition that can be phrased in terms of Schubert polynomials.
\end{enumerate}

Each pair (edge,permutation) produces a defining inequality for the polytope $\Lambda(d,N,\bm{w})$ \cite[Equation~(13)]{altunbulak_pauli_2008}.
Both steps are theoretical triumphs, however, in practice they retain a high computational complexity.
Our motivation for the present paper is to efficiently compute the first part of the solution.

The polytope $\Lambda(d,N,\bm{w})$ is not a lineup polytope, however it is closely related to the $ \pol[L](\Pi_{d,N}) $, the lineup polytope of the product of $ N $ simplices of dimension $ d-1 $.
There is a region which we call the \emph{test cone} (Equation \eqref{eq:testCone}) for which the support function of both polytopes agree.
This means that we can get some of the defining inequalities of $\Lambda(d,N,\bm{w})$ by means of computing $ \pol[L](\Pi_{d,N}) $.

\subsection{Examples}

\begin{example}[{$[n,m]$-grid}]\label{ex:grid}
Let $\mathbf{A}=\{1,2,\dots,n\}\times\{1,2,\dots,m\}\subset\R^2$ and set $r=mn$.
The number of $r$-lineups, or sweeps, is equal to the number of vertices of the sweep polytope of~$\mathbf{A}$.
Since this is a polygon, the number of vertices is equal to the number of edges and this corresponds to uncoarsenable rankings.
In this context, a ranking is uncoarsenable as long as the corresponding functional puts two points in a tie.
By symmetry, we can assume without loss of generality that $(1,1)$ comes first in the ranking.
So we must count the number of line segments with one endpoint in $\{(1,1),\dots,(n,1)\}$, the other in $\{(1,1),\dots,(1,m)\}$, up to parallel translations.
By counting the slopes in lowest fractional terms, we find that the number of parallel classes is
\[
\sum_{i=1}^n \phi(i,m) + 2 = \sum_{i=1}^m \phi(i,n) + 2,
\]
where $\phi(a,b)=|\{k\in[b]~:~\,\text{gcd}(k,a)=1\}|$.
The $+2$ comes from the segments parallel to the axes.
Considering symmetry, we have
\[
4\left(\sum_{i=1}^n \phi(i,m) \right)+2
\]
uncoarsenable rankings, and thus also the same number of sweeps.
For example, when $ (m, n) = (4, 3) $, the sweep polytope is the convex hull of $12!=479\ 001\ 600$ points. 
However, we obtain only 38 sweeps, so that the resulting sweep polygon has 38 vertices, and thus 38 edges.

\begin{figure}[!htbp]
\begin{tikzpicture}[scale=0.8]
\foreach \x in {0,...,3}
\foreach \y in {0,...,2} 
{
\draw (\x,0)--(0,\y);
}
\draw[red,line width=2pt] (1,0)--(0,1) (2,0)--(0,2);
\foreach \x in {0,...,3}
\foreach \y in {0,...,2} 
{
\filldraw (\x,\y) circle (1pt);
}
\end{tikzpicture}
\caption{8 line segments between the points in both axis but two of them are parallel so they correspond to the same ranking.}
\end{figure}
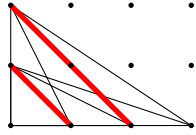
\end{example}

\begin{example}
Let $\pol$ be the prism over the two-dimensional triangle $ \mathsf{T} = \conv\{(0,0,0),(0,1,0),(0,0,1)\}$ and $\bw=(6,5,4,3,2,1)/21$.
The lineup polytope $\pol[L]_{6,\bw}(\pol)$ of $\pol$ is depicted in Figure~\ref{fig:lineup_prism}. 
\begin{figure}[!htbp]
\input{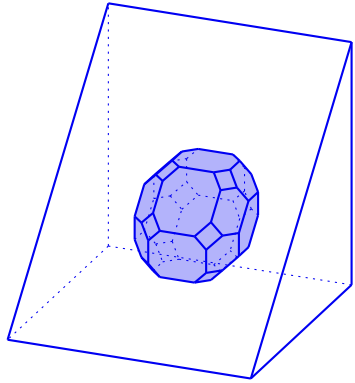}
\caption{The lineup polytope of the prism over a triangle with weights $(6,5,4,3,2,1)/21$.}
\label{fig:lineup_prism}
\end{figure}
\end{example}

\begin{example}[Standard simplex]\label{ex:simplex}
Let $ \mathbf{A} =  \{ \bm{e}_{1}, \dots, \bm{e}_{d} \}  $ be the canonical basis of $ \mathbb{R}^{d} $.
The convex hull $ \conv(A) $ is the \defn{standard simplex} of dimension $ (d-1) $ and it is denoted by $ \Delta_{d-1} $.
The sweep polytope $ \pol[L]_{d,\bw}(\mathbf{A}) $ in this case is equal to $ \text{conv} \{ w_{\pi(1)}, w_{\pi(2)}, \dots, w_{\pi(d)}| \pi \in \mathfrak{S}_d\} $, known as a \textit{permutohedron}, see \cite{postnikov_permutohedra_2009}.
\end{example}

\begin{example}[{Product of two segments, or the $[-1,1]$-square}]\label{ex:2qb}
The polytope $\Delta_1 \times \Delta_1$ is a square in $\R^2\times\R^2=\R^4$ with vertices \[\{(1,0;1,0),(1,0;0,1),(0,1;1,0),(0,1;0,1)\}.\]
To lower the dimension of the ambient space we map 
\[
	\gamma:\R^2\times\R^2\to\R^2,\qquad (x_{11},x_{12};x_{21},x_{22})\mapsto (x_{11}-x_{12} , x_{21}-x_{22}).
\]
Under this projection the polytope $\Delta_1 \times \Delta_1$ maps to the square $[-1,1]\times[-1,1]$ in $ \mathbb{R}^{2} $.
The 24 occupation vectors are illustrated in Figure \ref{fig:2qubits} along with their convex hull.
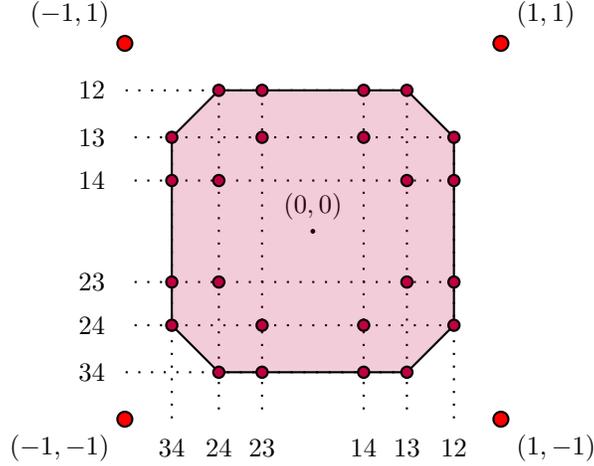
\begin{figure}[!htbp]
\begin{tikzpicture}
	[scale=2.5,
	vertex/.style={inner sep=1.5pt,circle,draw=black,fill=purple,thick},
	vertex_carre/.style={inner sep=2pt,circle,draw=black,fill=red,thick},
	axis/.style={loosely dotted, thick},
	poly/.style={draw=black,thick,fill=purple,fill opacity=0.2},
	baseline=(center)]
	
	\node[inner sep=0.5pt,circle,draw=black,fill=black,label=above:{$(0,0)$}] (center) at (0,0) {};

	\def\wun{0.63}
	\def\wdeux{49/200}
	\def\wtrois{0.12}
	\def\wquatre{0.005}
	
	\coordinate (1312) at (\wun-\wdeux+\wtrois-\wquatre,\wun+\wdeux-\wtrois-\wquatre) {};
	\coordinate (1213) at (\wun+\wdeux-\wtrois-\wquatre,\wun-\wdeux+\wtrois-\wquatre) {};
	\coordinate (1412) at (\wun-\wdeux-\wtrois+\wquatre,\wun+\wdeux-\wtrois-\wquatre) {};
	\coordinate (1214) at (\wun+\wdeux-\wtrois-\wquatre,\wun-\wdeux-\wtrois+\wquatre) {};
	\coordinate (1413) at (\wun-\wdeux-\wtrois+\wquatre,\wun-\wdeux+\wtrois-\wquatre) {}; 
	\coordinate (1314) at (\wun-\wdeux+\wtrois-\wquatre,\wun-\wdeux-\wtrois+\wquatre) {};
	
	\draw[axis] (1312) -- (\wun-\wdeux+\wtrois-\wquatre,-1) node[label=below:{$13$}] {};
	\draw[axis] (1213) -- (\wun+\wdeux-\wtrois-\wquatre,-1) node[label=below:{$12$}] {};
	\draw[axis] (1412) -- (\wun-\wdeux-\wtrois+\wquatre,-1) node[label=below:{$14$}] {};
	
	\draw[axis] (1312) -- (-1,\wun+\wdeux-\wtrois-\wquatre) node[label=left:{$12$}] {};
	\draw[axis] (1213) -- (-1,\wun-\wdeux+\wtrois-\wquatre) node[label=left:{$13$}] {};
	\draw[axis] (1214) -- (-1,\wun-\wdeux-\wtrois+\wquatre) node[label=left:{$14$}] {};
	
	\gettikzxy{(1312)}{\x}{\y}
	\draw[axis] (-1*\x,\y) -- (-1*\x,-1) node[label=below:{$24$}] {};
	\gettikzxy{(1213)}{\x}{\y}
	\draw[axis] (-1*\x,\y) -- (-1*\x,-1) node[label=below:{$34$}] {};
	\gettikzxy{(1412)}{\x}{\y}
	\draw[axis] (-1*\x,\y) -- (-1*\x,-1) node[label=below:{$23$}] {};
	
	\gettikzxy{(1312)}{\x}{\y}
	\draw[axis] (\x,-1*\y) -- (-1,-1*\y) node[label=left:{$34$}] {};
	\gettikzxy{(1213)}{\x}{\y}         
	\draw[axis] (\x,-1*\y) -- (-1,-1*\y) node[label=left:{$24$}] {};
	\gettikzxy{(1214)}{\x}{\y}         
	\draw[axis] (\x,-1*\y) -- (-1,-1*\y) node[label=left:{$23$}] {};

	\gettikzxy{(1312)}{\a}{\b}         
	\gettikzxy{(1213)}{\c}{\d}         
	\filldraw[poly] (1312) -- (1213) -- (\c,-\d) -- (\a,-\b) -- (-\a,-\b) -- (-\c,-\d) -- (-\c,\d) -- (-\a,\b) -- cycle;
	
	\foreach \a/\b in {1/1,1/-1,-1/1,-1/-1}{%
		
		\ifthenelse{\b=1}{\ifthenelse{\a=-1}{\node[vertex_carre,label=above left:{$(\a,\b)$}] at (\a,\b) {};}{\node[vertex_carre,label=above right:{$(\a,\b)$}] at (\a,\b) {};}}
		{\ifthenelse{\a=-1}{\node[vertex_carre,label=below left:{$(\a,\b)$}] at (\a,\b) {};}{\node[vertex_carre,label=below right:{$(\a,\b)$}] at (\a,\b) {};}}
		\gettikzxy{(1312)}{\x}{\y}
		\node[vertex] at (\a*\x,\b*\y) {};
		\gettikzxy{(1213)}{\x}{\y}
		\node[vertex] at (\a*\x,\b*\y) {};
		\gettikzxy{(1412)}{\x}{\y}
		\node[vertex] at (\a*\x,\b*\y) {};
		\gettikzxy{(1214)}{\x}{\y}
		\node[vertex] at (\a*\x,\b*\y) {};
		\gettikzxy{(1413)}{\x}{\y}
		\node[vertex] at (\a*\x,\b*\y) {};
		\gettikzxy{(1314)}{\x}{\y}
		\node[vertex] at (\a*\x,\b*\y) {};
	}
\end{tikzpicture}
\caption{The 24 possible points illustrated. 
The pair of numbers along the axis are shortcuts giving the indices of the $w$'s that are positive, while the others are negative. 
For instance the pair $14$ represents the number $w_1-w_2-w_3+w_4$.}
\label{fig:2qubits}
\end{figure}

The convex hull is the polyhedron consisting of all points $(x_1,x_2)\in\R^2$ satisfying the following linear inequalities:
\begin{equation}\label{eq:2bits_ineqs}
\begin{array}{rll}
-w_1-w_2+w_3+w_4\leq&x_1\phantom{+x_2}&\leq w_1+w_2-w_3-w_4,\\
-w_1-w_2+w_3+w_4\leq& \phantom{x_1+}\:\:x_2&\leq w_1+w_2-w_3-w_4,\\
-2w_1+2w_4\leq&x_1+x_2&\leq2w_1-2w_4,\\
-2w_1+2w_4\leq&x_1-x_2&\leq2w_1-2w_4.
\end{array}
\end{equation}

\end{example}

Recall from Section~\ref{ssec:conn_phys}, that for any vector $ \bm{x} \in \mathbb{R}^{d}$, we define $\bm{x}^{\downarrow}$ as the vector consisting of the absolute values of the entries of $\bm{x}$ in weakly decreasing order.
Using that notation, the $H$-representation of the sweep polytope $ \Sigma_{2,2} $ of the square $[-1,1]\times[-1,1]$ given in Equation~\eqref{eq:2bits_ineqs} can be rewritten as
\begin{equation*}
\Sigma_{2,2}=\left\{\bx\in\R^2:\quad
\left(
\begin{array}{rr}
1 & 0 \\
1 & 1 \\
\end{array}
\right)\bx^{\downarrow} \leq
\left(
\begin{array}{rrrr}
1 & 1 & -1 & -1\\
2 &  0 & 0 & -2\\
\end{array}
\right)\bw
\right\}.
\end{equation*}

\begin{example}[{Product of three segments, or the $[-1,1]$-cube}]\label{ex:3qb}
Let us now consider the sweep polytope of $(\Delta_1)^3$, the product of three line segments.
As in the previous example we map $(\Delta_1)^3$ to $\R^3$ by taking the difference on each factor.
The image is the cube $[-1,1]^3 \subset \mathbb{R}^{3} $.
This cube has 8 vertices, so if one computes the convex hull, $8!=40320$ points needs to be considered of which only $96$ forms the convex hull, see Figure~\ref{fig:3qubits}.
Algorithms~\ref{algo:test} and~\ref{algo:certify} determine directly these $96$ vertices.
	
\begin{figure}[!htbp]
\includegraphics[width=0.5\textwidth]{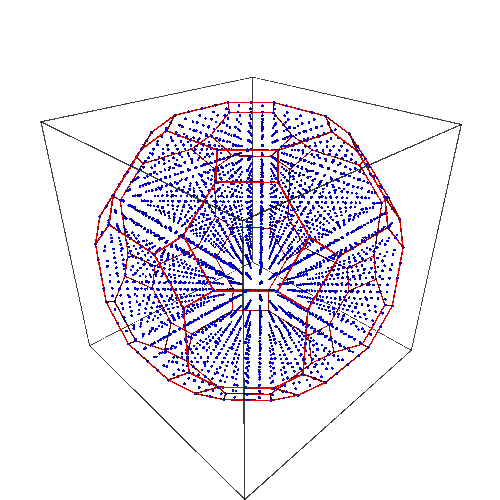}
\caption{The convex hull of all 40320 occupation vectors arising from all permutations of the vertices of the cube $[-1,1]^3$.}
\label{fig:3qubits}
\end{figure}

\noindent
Using the $^\downarrow$-notation, we can write the $H$-representation as
\begin{equation*}
\Sigma_{3,2}=\left\{\bx\in\R^3:\quad
\left(
\begin{array}{rrr}
1 & 0 & 0 \\
1 & 1 & 0 \\
1 & 1 & 1 \\
2 & 1 & 1 \\
\end{array}
\right)\bx^{\downarrow} \leq
\left(
\begin{array}{rrrrrrrr}
1 & 1 & 1 & 1 & -1 & -1 & -1 & -1 \\
2 & 2 & 0 & 0 &  0 &  0 & -2 & -2 \\
3 & 1 & 1 & 1 & -1 & -1 & -1 & -3 \\
4 & 2 & 2 & 0 &  0 & -2 & -2 & -4 \\
\end{array}
\right)\bw
\right\}.
\end{equation*}
\end{example}

\subsection{Certifying that a vector spans a ray}
In general, a vector $\by\in\R^d$ may not induce a total order on $\mathbf{A}$ as there may be ties.
Instead, the linear functional $\np{\by,\cdot}$ induces a \emph{ordered set partition} $\mathcal{S}=(S_1,\dots,S_{k-1},S_{k})$ of $[n]$, as follows:
\begin{itemize}
\item For each $ i = 1, \dots, k-1 $ the set $ S_i $ consists of labels of points where the functional achieves the $i$-th largest value.
\item We have $| S_1 \cup \dots \cup S_{k-2} | < r $ and $| S_1 \cup \dots \cup S_{k-1} | \geq r $.
\item The last set $ S_{k} $ consists of everything else.
\end{itemize}
We call such an ordered set partition induced by some $\by$ an \defn{$r$-ranking}.
Faces of the $r$-lineup polytope $\pol[L]_{r}(\mathbf{A})$ are in bijection with $r$-rankings.
We describe some linear programs that verify whether a given vector $ \bm{y} $ induces an uncoarsenable $r$-ranking or not.
For ease of notation we focus on the case where $ r $ is maximal (and we drop the ``$r$-'' from the name), but the propositions below can be readily adapted to the general setup.

\begin{proposition}\label{prop:lp_realizable}
Let $\mathbf{A}=\{\bv_1,\dots,\bv_n\}\subset\R^d$ and $\mathcal{S}=(S_1,S_2,\dots,S_k)$ be an ordered set partition of $[n]$.
Fix $\{i_1,\dots,i_k\} \subset [n] $ with $i_j\in S_j$ for $j=1,\dots,k$.
For each integer $t\in[k-1]$ consider the linear program given by
\[
\begin{array}{lrcl}
	\text{maximize} & \alpha_t,&&\\
	\text{subject to} & \langle\by,\bv_a-\bv_b\rangle &=&0 \text{ whenever }a,b\in S_k,\\
	&\langle\by,\bv_{i_{j+1}}-\bv_{i_j}\rangle&=&\alpha_{j}\text{ for }j=1,\dots,k-1,\\
	&(\bm{\alpha},\by)&\in&\R_{\geq 0}^{k-1}\times \R^d.
	
\end{array}
\]
The ordered set partition $\mathcal{S}$ is a ranking if and only if the $k-1$ linear programs above each have a positive solution.
\end{proposition}

\begin{proof}
If $ \mathcal{S} $ is a ranking, then there exists a vector $\by\in\R^d$ that whose inner products on consecutive blocks strictly increase.
This implies that there exist positive gaps $\alpha_1,\dots,\alpha_{k-1}$; proving the first direction.

Assume that the $k-1$ linear programs have non-zero solutions $(\alpha_1,\dots,\alpha_{k-1})$.
As the origin satisfy the inequalities, the linear programs are feasible and the solutions satisfy $\alpha_i\geq 0$ for $i\in[k-1]$.
A positive solution $\alpha_t>0$ implies that the ranking induced by the corresponding $\by_t$ is a coarsening of $\mathcal{S}$ with $S _t$ and $S_{t+1}$ in different blocks.
Therefore the vector $\by=\sum_t\by_t$ induces the common refinement of the induced rankings, which is equal to $\mathcal{S}$ by construction.
\end{proof}

Since the face poset $r$-lineup polytopes is isomorphic to the poset of $r$-rankings ordered by coarsening, the facets correspong to the $r$-rankings that cannot be coarsened by any other $r$-ranking.
We call such rankings geometrically \defn{uncoarsenable}.

\begin{proposition}\label{prop:lp_uncoarsenable}
Let $\mathbf{A}=\{\bv_1,\dots,\bv_n\}\subset\R^d$ be a point configuration and $\mathcal{S}=(S_1,S_2,\dots,S_k)$ be a ranking.
Fix a transversal $\{i_1,\dots,i_k\} \subset [n] $ with $i_j\in S_j$ for $j\in[k]$.
For each integer $t\in[k-1]$ consider the linear program given by
\[
\begin{array}{lrcl}
	\text{maximize} & \alpha_1+\dots+\alpha_{k-1},&&\\
	\text{subject to} & \langle\by,\bv_a-\bv_b\rangle &=&0 \text{ whenever }a,b\in S_k,\\
	&\langle\by,\bv_{i_{j+1}}-\bv_{i_j}\rangle&=&\alpha_{j}\text{ for }j=1,\dots,k-1,\\
	&\alpha_t&=&0\\
	&(\bm{\alpha},\by)&\in&\R_{\geq 0}^{k-1}\times \R^d.
		
\end{array}
\]
The ranking $\mathcal{S}$ is geometrically uncoarsenable if and only if zero is the solution to every $k-1$ linear programs above.
\end{proposition}
\begin{proof}
Assume that the $t$-th linear program has a positive solution.
Then, there exists a vector $\by_t$ whose induced ranking is a coarsening of $\mathcal{S}$ such that
(1) the points with indices in $S_t$ and $S_{t+1}$ are together in the same block, and (2) it has at least two blocks (since $\alpha_t>0$).
Therefore $\mathcal{S}$ is geometrically coarsenable.

If $\mathcal{S}$ has a nontrivial coarsening $\mathcal{R}$ given by some vector $\by$, the coarsening must
(1) contain at least two blocks, and (2) merge two consecutive blocks of $\mathcal{S}$.
Therefore, $\by$ provides a positive solution $\bm{\alpha}$ to a integer program for some $t\in[k-1]$.
\end{proof}

As described in \cite{pak_what_2022} combinatorial interpretations shall be understood as $\#\mathsf{P}$ counting problems.
Here, we prove that given a $n$ point configuration $\mathbf{A}$ and an ordered set partition $\mathcal{S}$ of $[n]$, the problem of determining whether $\mathcal{S}$ corresponds to a facet of $\pol[L]_{r,\bw}(\mathbf{A})$ is in $\mathsf{NP}$.

\begin{corollary}
The problem of counting the number of facets of the lineup polytope $\pol[L](\mathbf{A})$ is in $\#\mathsf{P}$.
\end{corollary}
\begin{proof}
Propositions \ref{prop:lp_realizable} and \ref{prop:lp_uncoarsenable} give a method to certify that an ordered set partition with $k$ parts is realizable and uncoarsenable.
This is done with $2k-2$ linear programs each of which has $n$ equalities, at most $k$ inequalities and $d+k-1$ variables.
Since each linear program can be solved in polynomial time \cite[Theorem 13.4]{schrijver_theory_1998} the conclusion follows.
\end{proof}

\section{General products of simplices}
\label{sec:gen_prod}

In this section, we study the lineup polytopes of product of simplices without any restrictions on their dimensions or their numbers.
In the following section, we consider the case of products of segments.
To ease the exposition, we treat the case where all simplices have the same dimension but everything extends to the general case.

\subsection{Combinatorial algorithms}
We first define an important poset that helps us navigate sweeps.
For a standard reference on posets, see \cite[Chapter~3]{stanley_enumerative_2012}.

\begin{definition}\label{def:youngLattice}
Let $\bi$ be the total order on the set $\{0,1,2,\dots,i\}$.
Given integers $d,N$, let $P(d,N) := \bm{d} \times \dots \times \bm{d}$ be the Cartesian product of $N$ copies of $\bm{d}$.
The \defn{$N$-dimensional Young lattice} $J(P(d,N))$ consists of lower order ideals of $P(d,N)$.
\end{definition}

By setting $N=2$ in Definition \ref{def:youngLattice}, we get the usual Young lattice on Young diagrams contained in a $(d+1) \times (d+1) $ box.
Let $\Delta_{d-1}$ be the $ (d - 1) $-dimensional regular simplex whose vertices are the canonical basis vectors of~$\R^d$.
Define $ \Pi_{d,N} \subset \mathbb{R}^{d \times N} $ as the Cartesian product $\Delta_{d-1}\times \dots\times \Delta_{d-1}$ of $N$ copies of $ \Delta_{d-1}$.
The vertices of~$\Pi_{d,N}$ are in natural bijection with the elements of the poset $P(d-1,N)$. 
In order to obtain the lineup polytope of $\Pi_{d,N} $, we use symmetry.
The group $ \mathfrak{S}_d \wr \mathfrak{S}_{N} $ acts naturally on $ \mathbb{R}^{d \times  N} $ and it fixes the polytope $\Pi_{d,N}$ and thus its normal fan.
Without loss of generality, we may restrict the set of linear functionals to the \defn{test cone}~$\mathsf{T}_{d,N}$ in $\mathbb{R}^{d\times N}$, whose coordinates are all positive and increase in each factor.
More precisely, 
\begin{equation}\label{eq:testCone}
\mathsf{T}_{d,N} := \left\{ \bm{x} \in \mathbb{R}^{d \times  N } ~\middle|~ 0 \leq x_{1,j} \leq \dots \leq x_{d,j} \text{ for }j\in[N]  \right\}.
\end{equation}
Due to the convention of ordering decreasingly the values to obtain a lineup, we momentarily need to consider \emph{upper} order ideals, for the following proposition to fit properly later.
Upper order ideals in finite posets are in natural correspondance with \emph{lower} order ideals.

\begin{proposition}\label{prop:orderIdeals}
Let $ \ell = (\bm{v}_{1}, \dots , \bm{v}_{r})$ be an $ r $-lineup of the vertices of $ \Pi_{d,N} $ induced by a linear functional $ \bm{y} \in \mathsf{T} _{d,N}$.
The corresponding elements of the poset $P(d,N)$ form a upper order ideal.
\end{proposition}

\begin{proof}
For any $ \bm{v} \in  \Pi_{d,N}  $ the defining linear inequalities of the test cone in Equation \eqref{eq:testCone} imply that all upper (in the poset) elements must have a larger value hence come earlier in the lineup.
\end{proof}

Since any initial segment of a lineup is itself a lineup, then Proposition \ref{prop:orderIdeals} gives us more refined information.
An $ r $-lineup of $ \Pi_{d,N} $ is a \textit{saturated chain of ideals of length $ r $ in the poset $J(P(d,N))$.}
However, the opposite is not true, see \cite[Section~9.3]{castillo_effective_2021} for some examples that may be extended here and also \cite[Introduction]{mallows_young_2015}.
Nonetheless, we use this correspondence to generate all potential lineups for which we may recursively verify whether they are \emph{realizable} or not using a set of inequalities in the test cone.
The set of vectors $\by$ that yield a certain lineup turns out to be the set of certificates of feasibility of LPs that Mallows and Vanderbei used to obtain so-called \emph{realizable} Young tableaux.
Recursively constructing the certificates provides an efficient method to enumerate the realizable Young tableaux instead of performing an LP for each Young tableau, see Section~\ref{ssec:prod_two}.
In \cite{castillo_effective_2021}, the potential lineups are refered to as \emph{shifted ideals} and the \emph{realizable} ones to the stricter notion of \emph{threshold ideals}.
To be correct, we should be referring to them as ``saturated chains of length $r$ exhausting a shifted or threshold ideal'', but that is rather wasteful.

We define $ \mathcal{L}^r_{d,N} $ the normal fan of the lineup polytope $\pol[L]_{r,\bw}(\Pi_{d,N})$, and as usual where drop $ r $ from the notation when it is maximal.
Instead of computing the complete normal fan, we instead compute its intersection with the test cone.
This process induces a fan supported on the test cone that we call the \defn{test fan} of $ \pol[L]_{r,\bw}(\Pi_{d,N}) $ and denote it by $\mathcal{T}^r_{d,N}$.
By restricting Equation \eqref{eq:refinement} to the test cone we obtain the sequence
\begin{equation*}
\mathsf{T}_{d,N}=\mathcal{T}^1_{d,N}  \succeq \mathcal{T}^{2}_{d,N} \succeq \dots \succeq \mathcal{T}_{d,N}.
\end{equation*}
When $r=1$, $ \mathcal{T}^1_{d,N} $ is simply the whole cone $\mathsf{T}_{d,N}$ and all of its faces.

\begin{proposition}\label{prop:partial}
Let $\tau$ be a ray of $\mathcal{T}^r_{d,N}$. 
The vector $\tau$ is a ray of $\mathcal{L}_{d,N}$, the fan of the sweep polytope~$\pol[L](\Pi_{d,N})$.
However, $\tau$ is not necessarily a ray of $\mathcal{L}^r_{d,N}$.
\end{proposition}

\begin{proof}
Since we are refining the fan at each step it means that $ \tau $ is a ray of $ \mathcal{T}_{d,N} $.
In the sweep polytope~$\pol[L](\Pi_{d,N})$, the interior of every lineup cone is either contained in the test cone $ \mathsf{T}_{d,N} $ or disjoint from it.
Indeed, a sweep induced by an element not in the test cone is necessarily different from one induced by a test element.
This means that that normal fan of the sweep polytope intersected with the test cone, i.e. its test fan, is simply the set of the lineup cones contained in the test cone; there are no new cones generated by the intersection with $\mathsf{T}_{d,N}$.
It follows that $\tau$ is a ray of $\mathcal{L}_{d,N}$ as sought.
\end{proof}

The algorithm described below is an adjusted version of \cite[Algorithm F]{castillo_effective_2021}.
It yields all $r$-lineups of products of simplices.
More precisely, we compute the lineup cones, as we need a certificate that an ordered list of $r$ points is indeed a lineup.
To obtain the lineup cones, we restrict the computation to their intersection with the test cone to consider each orbit exactly once and use recursion on $r$. 
The set of candidates to append to an $r$-lineup is taken from Proposition~\ref{prop:orderIdeals}.

\begin{algorithm}[Recursively construct lineups]
\label{algo:test}
Let $r,d,N$ be positive integers such that $ 1 \leq r \leq d^N $.
The following procedure compute all $r$-lineups cones of the polytope $\Pi_{d,N}$ that intersect the interior of the test cone $\mathsf{T}_{d,N}$.

\begin{description} 
\item[Base case, $r=1$] If the vector $\by$ is in the test cone, then there is only one possibility: the normal cone of the vertex $ \bm{e}_d \times \dots \times \bm{e}_{d} $ intersected with the test cone is equal to the complete cone $\mathsf{T}_{d,N}$ whose $H$-representation is given in Equation \eqref{eq:testCone}.

\item[Inductive step, $r>1$] Having all $r$-lineups together with their normal cones in $\mathsf{T}_{d,N}$, we proceed to obtain the $(r+1)$-lineups.
For each $r$-lineup $\ell$ cone $\mathsf{K}$, the possible candidates for being in position $r+1$ are limited by the partial order on $J(P(d,N))$.
Say there are candidates $C\subseteq J(P(d,N)) $, then $\bi$ is allowed to be the next one if and only if the cone
\begin{equation} \label{eq:r+1}
\mathsf{K} \cap \left\{ \bm{y} \in \prod_{j=1}^{N} \mathbb{R}^{d} ~:~ \np{ \bm{y} , \bm{v}(\bi) } \geq \np{ \bm{y} , \bm{v}(\bj) }, \text{ for all } \bj\in C\setminus\{\bi\} \right\}
\end{equation}
has the same dimension as that of $\mathsf{K}$.
If so, then $ ( \ell,\bi_1 ) $ is an $r$-lineup and Equation~\eqref{eq:r+1} describes its corresponding normal cone. 
Otherwise, we discard it.
	
\end{description}
\end{algorithm}

\begin{algorithm}[Reduction step to obtain facets]
\label{algo:certify}
Having an $H$-representation in Equation~\eqref{eq:r+1}, we translate it into a $V$-representation to obtain its extremal rays.
These rays form the set of potential facet inequalities for a given lineup.
\begin{description}
\item[Certifying rays] 
We can assert that the extremal rays are indeed rays of the lineup fan (and not a product of intersecting with the test cone) by computing the induced ranking and using the linear program in Proposition \ref{prop:lp_uncoarsenable} to verify that the ranking is uncoarsenable.
\end{description}
\end{algorithm}

Algorithm~\ref{algo:test} computes the complete list of $r$-lineups together with certificates of their existence (actually the set of all certificates which is the lineup fan).
If we want to have the rays of the fan, it is better to keep both representations at every step of Algorithm~\ref{algo:test} prior to apply Algorithm~\ref{algo:certify}.

\begin{remark} \label{rem:assertRays}
If we are interested in rays of the sweep fan, then asserting rays in the last step by Proposition~\ref{prop:partial} is not necessary.
Even though we may obtain non-facet-defining rays for the $ r $-lineup fans with the previous steps, all the sweep cones are subsets of the test cone so any ray appearing there is an actual ray of the last fan, i.e. the sweep fan.
\end{remark}

\begin{question}
Finally, we wonder whether it is possible to characterize the family of symmetric polytopes for which an adapted recursive procedure such as the one above exists.
As far as we know, this procedure works for hypersimplices, dilated simplices and products of them.
Which other operations on polytopes could be done?
\end{question}

\subsection{Examples}

\subsubsection{Product of two simplices}
\label{ssec:prod_two}

Using the test cone on the product of two simplices $\Delta_e\times\Delta_f$, we restrict ourselves to linear functionals of the form $(\bm{a},\bm{b})$ where
\[
0=a_1\leq a_2\leq\dots\leq a_{e+1},\qquad \text{and} \qquad 0=b_1\leq b_2\leq\dots\leq b_{f+1}.
\]
The values of the functional $(\bm{a},\bm{b})$ on the vertices of  $\Delta_e\times\Delta_f$ are the elements of the set $\{a_i+b_j~:~(i,j)\in[e+1]\times[f+1]\}$.
We may arrange them in a $(e+1)\times (f+1)$ tableau, and replace the entries by their relative order from 1 to $\bm{a}\bm{b}$.
\begin{example}
For example let $e=2,f=3$ and the linear functional $(\bm{a},\bm{b})=(0,1,7;0,3,8,11)$ then the tableaux are given in Figure \ref{fig:tableau_example}.
\begin{figure}[ht]
\begin{tikzpicture}[scale=0.6]
\node at (2,1.5) {$\Longrightarrow$};
\begin{scope}[xshift=-4cm]
\draw[black,thin] (0,0) grid (4,3);
\node[left] at (0,2.5) {$\bm{0}$};
\node[left] at (0,1.5) {$\bm{1}$};
\node[left] at (0,0.5) {$\bm{7}$};
\node[above] at (0.5,3) {$\bm{0}$};
\node[above] at (1.5,3) {$\bm{3}$};
\node[above] at (2.5,3) {$\bm{8}$};
\node[above] at (3.5,3) {$\bm{11}$};

\node at (0.5,2.5) {0};
\node at (1.5,2.5) {3};
\node at (2.5,2.5) {8};
\node at (3.5,2.5) {11};
\node at (0.5,1.5) {1};
\node at (1.5,1.5) {4};
\node at (2.5,1.5) {9};
\node at (3.5,1.5) {12};
\node at (0.5,0.5) {7};
\node at (1.5,0.5) {10};
\node at (2.5,0.5) {15};
\node at (3.5,0.5) {18};
\end{scope}
\begin{scope}[xshift=4cm]
	\draw[black,thin] (0,0) grid (4,3);
	
	\node at (0.5,2.5) {1};
	\node at (1.5,2.5) {3};
	\node at (2.5,2.5) {6};
	\node at (3.5,2.5) {9};
	\node at (0.5,1.5) {2};
	\node at (1.5,1.5) {4};
	\node at (2.5,1.5) {7};
	\node at (3.5,1.5) {10};
	\node at (0.5,0.5) {5};
	\node at (1.5,0.5) {8};
	\node at (2.5,0.5) {11};
	\node at (3.5,0.5) {12};
\end{scope}
\end{tikzpicture}
\caption{Values on the left and their relative positions on the right.}
\label{fig:tableau_example}
\end{figure}
\end{example}

For any given vector $(\bm{a},\bm{b})$, we consider the tableau of relative orders.
By construction, this tableau is weakly increasing along rows and columns.
Furthermore, it has the following properties
\begin{align}
\text{if } T(i,j)=T(i,j+1) \text{ for some }i\in[e+1], &\text{ then } T(i,j)=T(i,j+1) \text{ for all }i\in[e+1] \label{eq:row}\\
\text{if } T(i,j)=T(i+1,j) \text{ for some }j\in[f+1], &\text{ then } T(i,j)=T(i+1,j) \text{ for all }i\in[f+1] \label{eq:column}
\end{align}
We call a tableau satisfying Equations \eqref{eq:row}-\eqref{eq:column} a \emph{constrained Young tableau}.
If the vector is generic there are no ties and the associated tableau of relative orders is a \emph{standard Young tableau} or SYT for short.
A tableau is called \defn{realizable} if it is induced by a sweep; the definition originates from \cite{mallows_young_2015}.

\begin{example}
\label{ex:d=0}
When $e=0$ the product is isomorphic to $\Delta_f$.
In this case there is only one possible SYT and it is of course realizable.
It follows that for simplex any ordering of its vertices is a sweep.
Its sweep polytope is the permutohedron $\text{Perm}(\bw)$.
As long as $\bw$ is strictly decreasing the rays of its normal fan in the test cone are given by $(\underbrace{0,\dots,0}_{k},\underbrace{1,\dots,1}_{f+1-k})$, for all $k\in[e]$.
\end{example}

\noindent
\textbf{The case $e=1$.} 
In this case, every standard Young tableau is realizable, as was noted by Klyachko in \cite[Example 4.1.2]{klyachko_quantum_2004}.
It also appears in \cite[Section 2]{mallows_young_2015} and in \cite[Theorem 7.10]{black_flag_2022} in connection with monotone path polytopes.
We prove it once again to cover the case of rankings, not just sweeps.
\begin{proposition}
Every constrained Young tableau of size $2\times m$ is realizable.
\end{proposition}
\begin{proof}
We can assume $T(1,k+1)<T(2,k+1)$ for each $k$ since if they are equal then by Equation \eqref{eq:column} then the first row is equal to the second and any constrained tableau of one row is realizable, so we can simply set $\bm{a}=(0,0)$ and use a $\bm{b}$ that realizes the first row.

We set $\bm{a}=(0,1)$ and we define $\bm{b}$ one step an the time.
Start with $b_1=0$.
Assuming $b_1,\dots,b_k$ has been chosen so that the induced order is given by the relative order of first $k$ columns of the tableau we now choose $b_{k+1}$.
Let $\epsilon$ be small enough.
\begin{itemize}
	\item If $T(1,k+1)=T(1,k)$, then by Equation \eqref{eq:row} we have $T(2,k+1)=T(2,k)$ and we set $b_{k+1}=1+b_{k}$.
	\item Else, if $T(1,k+1)=T(2,j)$ for some $j\leq k$, then we set $b_{k+1}=1+b_j$.
	\item Else $T(1,k+1)$ does not appear in the first $k$ columns. Let $T(i,j)$ be the rightmost entry in the second row such that $T(i,j)<T(1,k+1)$.
	We set $b_{k+1}=1+b_j+\epsilon^i$, where $i$ is the smallest power of $\epsilon$ we have not used until this point.\qedhere
\end{itemize}
\end{proof}
It follows that the number of sweeps is equal to the number of SYT of shape $(2,n)$, which is equal to the $n$th Catalan number.
\begin{corollary}
The sweep polytope $\pol[L](\Delta_1\times\Delta_{f})$ has $C_{f+1}$ orbits of vertices.
\end{corollary}
Furthermore we can also describe the set of facets.
The following proposition is mentioned without proof in \cite[Example 4.2.1]{klyachko_quantum_2004}.
\begin{proposition}\label{prop:linearSYT}
The rays of the normal fan of $\pol[L](\Delta_1\times\Delta_f)$ inside the test cone are
\begin{equation*}
\begin{array}{cl}
\left(0,0~;~\underbrace{0,\dots,0}_{k},\underbrace{1,\dots,1}_{f+1-k}\right),& \text{ for all }k\in{[f]},\text{ and }\\ \left(0,1~;~\underbrace{0,\dots,0}_{p_1},\underbrace{1,\dots,1}_{p_2},\dots,\underbrace{k-1,\dots,k-1}_{p_k}\right),&\text{ for all partitions }p\vdash (f+1).
\end{array}	
\end{equation*}
\end{proposition}
\begin{proof}
We can assume that $b_1>0$, otherwise we are in the situation with one row and the first set of vectors arise from Example \ref{ex:d=0}.
	
The rays in the test cone correspond to the uncoarsenable rankings which we now describe.
Let $T$ be the induced tableau.
As mentioned above we can assume that its rows are distinct.
Even more, we can assume that it does not have two equal columns, as this case can be reduced to one with a smaller $f$.
So each entry appears at most twice in $T$.

Suppose the number $k$ appears only once and in the first row.
By the construction of Proposition \ref{prop:linearSYT}, we get a coarsening of $T$ by merging the blocks $k$ and $k-1$.
It follows that in an uncoarsenable ranking represented by $T$ every element of the first row, except for $1$ appears also in the second row and this implies that $T$ is equal to
\[
S(n)=\begin{array}{|c|c|c|c|c|}
	\hline
	1&2&3&\dots&n\\
	\hline
	2&3&4&\dots&n+1\\
	\hline
\end{array}
\]
In conclusion, every uncoarsenable tableau with different rows is equal to $S(n)$ and every tableau obtained from $S(n)$ by duplicating columns.
\end{proof}

\begin{example}
The list of uncoarsenable tableaux of size $2\times 3$ is
\[
\begin{array}{|c|c|c|}
	\hline
	1&2&3\\
	\hline
	2&3&4\\
	\hline
\end{array}, \quad
\begin{array}{|c|c|c|}
	\hline
	1&1&2\\
	\hline
	2&2&3\\
	\hline
\end{array},\quad 
\begin{array}{|c|c|c|}
	\hline
	1&2&2\\
	\hline
	2&3&3\\
	\hline
\end{array}, \quad 
\begin{array}{|c|c|c|}
\hline
1&1&2\\
\hline
1&1&2\\
\hline
\end{array}, \quad 
\begin{array}{|c|c|c|}
\hline
1&2&2\\
\hline
1&2&2\\
\hline
\end{array}.
\]
\end{example}

\noindent
\textbf{The case $e\geq 2$}.
The situation is more complicated as not every tableau is realizable.
Here we have the situation for $e=2$ and different values of $f$.
\[
\begin{array}{r|c|c|c|c|c|c|c|c}
\text{Number of columns}&1&2&3&4&5&6&7&8 \\
\hline
\text{Realizable SYT}&1&5&36&295&2583&23580&221680&2130493\\
\hline
\text{Total SYT}&1&5&42&462&6006&87516&1385670&23371634 
\end{array}
\]
As proved in \cite[Corollary 1.4]{araujo_realizable_2023}, the ratio between realizable tableaux and all tableaux tend to $0$ as the number of column increases.
In the following table, we record the total number of realizable standard tableaux of size $m\times n$ that we computed using Algorithm~\ref{algo:test}.
The bold entries are new contributions to the OEIS sequence \cite[A211400]{oeis}.

\begin{center}
\resizebox{\textwidth}{!}{$
\begin{tabular}{ c c|r r r r r r r r r r}
       & $ m $ &       &       &       &       &       &         &        &                     &  \\
 $n $  &       & 1     & 2     & 3     & 4     & 5     & 6       & 7      & 8                   & 9 & 10\\\hline
       & 1     & 1     & 1     & 1     & 1     & 1     & 1       & 1      & 1                   & 1 & 1\\
       & 2     &       & 2     & 5     & 14    & 42    & 132     & 429    & 1430 & 4862 & 16796 \\
       & 3     &       &       & 36    & 295   & 2 583  & 23 580   & 221 680 & {\bf\color{red}2 130 493}& {\bf\color{red}20 829 605} & {\bf\color{red}206 452 585}\\
       & 4     &       &      &    & 6 660  & {\bf\color{red}152 933} & {\bf\color{red}3 533 808} & {\bf\color{red} 81 937 118}       &                     &  \\
       & 5     &       &       &   &       & {\bf\color{red}8 499 376}      & {\bf\color{red} 449 879 088}        &        &                     &  \\
\end{tabular}$}
\end{center}
For $n=2$, we have the Catalan numbers which grow as $O(4^n/n^{3/2})$.
For $n=3$, the least-square log-fit provides a rate of growth of $\approx\!\!9.2808$.
For $n=4$, the least-square log-fit provides a rate of growth of $\approx\!\!23.0874$.
For $(3,10)$, the computations took $\approx\!\!93$cpudays split on a basic parallel map-reduce procedure in Sagemath to only count the number of lineups.
For $(4,7)$, the computations took $\approx\!\!31$cpudays.
For $(5,6)$, the computations took $\approx\!\!199$cpudays.

\subsubsection{Product of three simplices}

We computed the number of lineups for the following cases:

\begin{center}
\begin{tabular}{c|r}
Sizes & Number of lineup orbits \\\hline
$2\times2\times 2$ & 12 \\
$2\times2\times 3$ & 110 \\
$2\times3\times 3$ & 3 792 \\
$3\times3\times 3$ & 566 616 \\
$3\times3\times 4$ & 80 638 740 \\
\end{tabular}
\end{center}

\noindent
The first cases being that of Example~\ref{ex:3qb}.

\section{The case $d-1=1$: hypercubes}
\label{sec:hypercubes}

In the case of the products of line segments---that is of hypercubes---the associated group of symmetries is $ \mathfrak{S}_2 \wr \mathfrak{S}_{N} $ which enjoys further properties.
This is the \defn{hyperoctahedral group} or Coxeter group of type $B$.
We can refine the test cone and the Gale order previously defined.

\begin{definition}

The \defn{(extended) Gale order} $\text{G}(N)$ is the refinement of the Boolean lattice $2^{[N]}=\bm{1}\times\dots\times\bm{1}$ given by the following relation.
Recall that an element $S\in\bm{1}\times\cdots\times\bm{1}$ corresponds to a subset $S\subseteq [N]$.
Given two subsets $S=\{s_1,\dots,s_i\}$ and $T=\{t_1,\dots,t_j\}$, with elements ordered from smallest to largest, we say that $ S\leq T$ if and only if $|S|\leq|T|$ and $s_{i-k}\leq t_{j-k},$ for all $k\in\{0,\dots,i-1\}$, see Figure~\ref{fig:gale} for an illustration with $n=4$.
\end{definition}

\begin{remark}\label{rem:stan}
The poset $\text{G}(N)$ is the poset $M(n)$ of minimal coset representatives of the subgroup of usual permutations $\mathfrak{S}_n$ within the Coxeter group of signed permutations $B_n$, see \cite[Figure~6]{stanley_weyl_1980}.
When this order is restricted to subsets of a fixed cardinality $k$, one recovers the traditional Gale poset on $k$-subsets of~$[n]$, see Figure~\ref{fig:gale} for an example with $n=4$ and $k=2$.
\end{remark}

We also define a refinement of the test cone (see Equation \eqref{eq:testCone}) by ordering the gaps between the vectors on each factor.
\begin{equation*}
	\mathsf{F}_{N,2} := \left\{ \bm{x} \in \mathbb{R}^{2 \times  N } ~\middle|~ 0 \leq x_{1,j} \leq x_{2,j} \text{ for }j\in[N] \text{ and } (x_{2,i}-x_{1,i}) \leq  (x_{2,k}-x_{1,k}) \text{ whenever }i>k\right\}.
\end{equation*}

Furthermore, following Examples \ref{ex:2qb} and \ref{ex:3qb}, we lower the dimension of the ambient space by considering the differences
\begin{equation*}
	\gamma: \mathbb{R}^{2 \times N} \to \mathbb{R}^{N}, \qquad (x_{1,j}, x_{2,j})_{j \in N} \mapsto (x_{2,j} - x_{1,j})_{j \in N} .
\end{equation*}

The map $ \gamma $ induces a linear isomorphism between the product of simplices $ \Delta_1^N \subset \mathbb{R}^{2 \times  N} $ and the cube  $ [-1,1]^{N} \subset \mathbb{R}^{N} $.
The image, under $ \gamma $, of the refined test cone $\mathsf{F}_{N,2} $ is the \defn{fundamental chamber}:

\begin{equation*}
\Phi_N = \left\{ \bm{x} \in \mathbb{R}^{N} ~:~ x_1\leq x_2 \leq \dots \leq x_N \right\}.
\end{equation*}

\begin{figure}[!htbp]
	\begin{tikzpicture}%
		[scale=1.000000,
		edge/.style={dotted, thick},
		edgefull/.style={thick},
		vertex/.style={inner sep=1pt,circle,draw=black,fill=black,thick},
		vertex2/.style={inner sep=1.5pt,circle,draw=black,fill=white,thick},
		vertex3/.style={inner sep=2pt,rectangle,draw=black,fill=white,thick},
		rotate=-135]
		
		\coordinate (1234) at (0,0);
		\coordinate (123) at (1,0);
		\coordinate (124) at (2,0);
		\coordinate (134) at (3,0);
		\coordinate (234) at (4,0);
		\coordinate (12) at (2,-1);
		\coordinate (13) at (3,-1);
		\coordinate (14) at (3,-2);
		\coordinate (23) at (4,-1);
		\coordinate (24) at (4,-2);
		\coordinate (34) at (5,-2);
		\coordinate (1) at (3,-3);
		\coordinate (2) at (4,-3);
		\coordinate (3) at (5,-3);
		\coordinate (4) at (6,-3);
		\coordinate (0) at (7,-3);

		\node[vertex,label=1234] at (1234) {};
		\node[vertex,label=234] at (123) {};
		\node[vertex,label=134] at (124) {};
		\node[vertex,label=124] at (134) {};
		\node[vertex,label=123] at (234) {};
		\node[vertex,label=34] at (12) {};
		\node[vertex,label=24] at (13) {};
		\node[vertex,label=14] at (14) {};
		\node[vertex,label=23] at (23) {};
		\node[vertex,label=13] at (24) {};
		\node[vertex,label=12] at (34) {};
		\node[vertex,label=4] at (1) {};
		\node[vertex,label=3] at (2) {};
		\node[vertex,label=2] at (3) {};
		\node[vertex,label=1] at (4) {};
		\node[vertex,label=$\varnothing$] at (0) {};

		\draw[edge] (1234)--(234);
		\draw[edgefull] (12)--(23);
		\draw[edgefull] (14)--(34);
		\draw[edgefull] (14)--(13);
		\draw[edgefull] (24)--(23);
		\draw[edge] (1)--(0);
		\draw[edge] (124)--(12);
		\draw[edge] (134)--(13);
		\draw[edge] (234)--(23);
		\draw[edge] (1)--(14);
		\draw[edge] (2)--(24);
		\draw[edge] (3)--(34);
		
	\end{tikzpicture}
	\caption{Extended Gale $\text{G}(N)$ for $N=4$ and the Gale poset on $2$-subsets is illustrated with filled edges.}
	\label{fig:gale}
\end{figure}
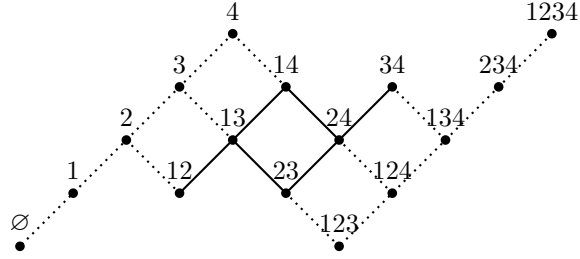

\begin{example}\label{ex:4qb}
Let $n=4$ and consider the sweep polytope of $[-1,1]^4$.
As opposed to the case $n=3$ shown in Example \ref{ex:3qb}, in this case brute force leads us nowhere, so we must rely on Algorithms~\ref{algo:test} and \ref{algo:certify}.
In this case, we obtain the following facet inequalities:

\begin{equation*}
\resizebox{\textwidth}{!}{$
\Sigma_{4,2}=\left\{\bx\in\R^4:\quad
\left(
\begin{array}{rrrr}
1 & 0 & 0 & 0 \\ 
1 & 1 & 0 & 0 \\ 
1 & 1 & 1 & 0 \\ 
1 & 1 & 1 & 1 \\ 
2 & 1 & 1 & 0 \\ 
2 & 1 & 1 & 1 \\ 
2 & 2 & 1 & 1 \\ 
3 & 1 & 1 & 1 \\ 
3 & 2 & 1 & 1 \\ 
3 & 2 & 2 & 1 \\ 
4 & 2 & 1 & 1 \\ 
4 & 3 & 2 & 1 \\ 
\end{array}
\right)\bx^{\downarrow} \leq
\left(
\begin{array}{rrrrrrrrrrrrrrrr}
1  & 1 & 1 & 1 & 1 & 1 & 1 & 1 & -1 & -1 & -1 & -1 & -1 & -1 & -1 & -1 \\ 
2  & 2 & 2 & 2 & 0 & 0 & 0 & 0 &  0 &  0 &  0 &  0 & -2 & -2 & -2 & -2 \\
3  & 3 & 1 & 1 & 1 & 1 & 1 & 1 & -1 & -1 & -1 & -1 & -1 & -1 & -3 & -3 \\
4  & 2 & 2 & 2 & 2 & 0 & 0 & 0 &  0 &  0 &  0 & -2 & -2 & -2 & -2 & -4 \\
4  & 4 & 2 & 2 & 2 & 2 & 0 & 0 &  0 &  0 & -2 & -2 & -2 & -2 & -4 & -4 \\
5  & 3 & 3 & 3 & 1 & 1 & 1 & 1 & -1 & -1 & -1 & -1 & -3 & -3 & -3 & -5 \\
6  & 4 & 4 & 2 & 2 & 2 & 0 & 0 &  0 &  0 & -2 & -2 & -2 & -4 & -4 & -6 \\
6  & 4 & 4 & 4 & 2 & 2 & 2 & 0 &  0 & -2 & -2 & -2 & -4 & -4 & -4 & -6 \\
7  & 5 & 5 & 3 & 3 & 1 & 1 & 1 & -1 & -1 & -1 & -3 & -3 & -5 & -5 & -7 \\
8  & 6 & 4 & 4 & 2 & 2 & 2 & 0 &  0 & -2 & -2 & -2 & -4 & -4 & -6 & -8 \\
8  & 6 & 6 & 4 & 4 & 2 & 2 & 0 &  0 & -2 & -2 & -4 & -4 & -6 & -6 & -8 \\
10 & 8 & 6 & 4 & 4 & 2 & 2 & 0 &  0 & -2 & -2 & -4 & -4 & -6 & -8 & -10 \\
\end{array}
\right)\bw
\right\}$.}
\end{equation*}
\end{example}

The length of a sweep of the set $ [-1,1]^N $ is $ 2^N $ and the number of sweeps, even when restricted to sweeps induced by functionals in the fundamental chamber, grows too fast for practical purposes.
However, Remark~\ref{rem:partial} shows that we can obtain some of the inequalities for the sweep polytope by considering $r$-lineup polytopes.

\begin{remark}
\label{rem:fermions}
In \cite{castillo_effective_2021} we studied the case of indistinguishable fermionic particles: mathematically we considered a single Hilbert space $ \mathcal{H} $ of dimension $d$ and took its $N$ antisymmetric power $ \bigwedge^N \mathcal{H} $.
Combinatorially, this led us to study the sweep polytope of the \textit{hypersimplex} $ \Delta(d,N) $.
For the case of $N$ distinguishable qbits (so $d=2$) the analogous polytope is the $N$-dimensional cube.
\end{remark}

Whereas in \cite{castillo_effective_2021}, we considered $d$ as a variable that was meant to grow to infinity, here for most applications a fixed $d$ suffices.
For practical purposes, what we need is to analyze what happens when $  N $ grows.

\begin{proposition}\label{prop:lift_N}
Let $r\geq2$ and $N\geq r-1$.
Suppose we have the $H$-representation  
\[
\pol[L]_{r}([-1,1]^N) =\{\bx\in\mathbb{R}^N~:~\np{\by,\bx^\downarrow}\leq b, (\by,b)\in I\},
\]
where $I\subseteq \Phi_N\times\mathbb{R}$.
For $M>N$ we have the $H$-representation
\[
\pol[L]_{r}([-1,1]^M) =\{\bx\in\mathbb{R}^M~:~\np{\by',\bx^\downarrow}\leq b', (\by,b)\in I\},
\]
where $\by'$ is obtained from $\by$ by appending $M-N$ entries equal to $y_1$ at the beginning and $b'=b+(M-N)y_1$.
With the exception that if $\by=(1,0,\dots,0)$ then $\by'=(1,0,0,\dots,0)$.
\end{proposition}

\begin{proof}
The set of upper ideals of length $ r $ of the poset $\text{G}(N)$ is isomorphic to that of $\text{G}(M)$ as long as $ N \geq 1 $.
The isomorphism is simply adding a 0 to every set and adding one to all elements.
When running Algorithm~\ref{algo:test} then we get the same steps and comparisons, with the only difference that the corresponding vectors have one more entry at the start, hence the replication of the first entry $ y_{1} $.
To obtain the new right hand side $ b' $ we apply the last part of Proposition \ref{prop:main_structure}.
\end{proof}

For example, consider the following inequality for $N= 4$ which is minimal for $r=4$,
\begin{equation*}
	(2, 2, 1, 1)\bx^{\downarrow} \leq (6,0,0,-2)\bw,
\end{equation*}
where $\bw=(w_1,w_2,w_3,w_4)$.
It is transformed into the following inequality which is valid for $N=7$ particles and $r=4$:
\begin{equation*}
	(\mathbf{2},\mathbf{2},\mathbf{2}, 2, 2, 1, 1)\bx^{\downarrow} \leq (6,0,0,-2)\bw + 2(7-4).
\end{equation*}

\begin{remark}\label{rem:white_whale}
The sweep polytope of $[-1,1]^n$ is normally equivalent to the sweep polytope of $[0,1]^n$.
This is a zonotope obtained by Minkowski adding the vectors $
\sum_{\bv_1,\bv_2\in\text{Vert}([0,1]^n)} \left(\bv_1-\bv_2\right).$
The \emph{resonance arrangement} is the hyperplane arrangement associated to the zonotope given by the Minkowski sum $\sum_{\bv\in\text{Vert}([0,1]^n)} \left(\mathbf{0}, \bv\right).$ Dubbed the \emph{White Whale} by Billera, some recent research have focused on the computations of its vertices, see \cite{deza_sizing_2022} and \cite{brysiewicz_enumerating_2021}.
Since $ \mathbf{0} $ is a vertex of $ [0,1]^N $, the White Whale is a Minkowski summand of the sweep polytope of the cube.
In other words, there exists a polytope $ Q $ such that
\begin{equation}\label{eq:white_whale}
 Q + \sum_{\bv\in\text{Vert}([0,1]^n)} \left(\mathbf{0}, \bv\right)= 
\sum_{\bv_1,\bv_2\in\text{Vert}([0,1]^n)} \left(\bv_1-\bv_2\right).
\end{equation}
With a mix of clever ideas and lots of computation power, the total number of vertices of the white whale is known only until $N=9$, see \cite{deza_sizing_2022}.
Equation \eqref{eq:white_whale} puts some context the hardness of computing the sweep polytope of the cube.
\end{remark}

\section{Further questions}
\label{sec:further}
\subsection{Relationship with the quantum marginal polytope}
Even though the polytope $\Lambda(d,N,\bm{w})$ of Equation~\eqref{eq:klypo} is not a sweep polytope, it contains some of the occupation vectors coming from lineups, namely the ones in the positive orthant.
We compare $\Lambda(d,N,\bm{w})$ and $\pol[L]_{r,\bm{w}}(\Pi_{d,N})$ in the following small example.

\begin{example}
Continuing Example \ref{ex:2qb}, without loss of generality, we may restrict the study of $\Sigma_{2,2} $ to the positive quadrant.
In Example \ref{ex:2qb} we computed $ \Sigma_{2,2,\bm{w}} $ the sweep polytope of the product of $ N = 2 $ simplices of dimension $ d-1 $ with $ d = 2 $.
The inequalities obtained by Bravyi \cite{bravyi_requirements_2003} for the polytope $ \Lambda(2,2,\bm{w}) $ are as follows:
\[
\begin{array}{rcl}
0 \leq&x_1\phantom{+x_2}&\leq w_1+w_2-w_3-w_4,\\
0 \leq&\phantom{x_1+}\:\:x_2&\leq w_1+w_2-w_3-w_4,\\
0 \leq&x_1+x_2&\leq2w_1-2w_4,\\
-2\min\{w_1-w_3,w_2-w_4\}\leq&x_1-x_2&\leq2\min\{w_1-w_3,w_2-w_4\}.
\end{array}
\]
The last two inequalities are the only ones that are not obtained by restricting $ \Sigma_{2,2,\bm{w}} $ to the positive quadrant.
The comparison between $ \Sigma_{2,2,\bm{w}} $ restricted to the positive orthant and $ \Lambda(2,2,\bm{w}) $ is illustrated in Figure~\ref{fig:2qubits_bravi}, where $\mu:=2\min\{w_1-w_3,w_2-w_4\}$.

\begin{figure}[!htbp]
\begin{tikzpicture}
[scale=4,
vertex/.style={inner sep=1.5pt,circle,draw=black,fill=purple,thick},
vertex_carre/.style={inner sep=2pt,circle,draw=black,fill=red,thick},
axis/.style={loosely dotted, thick},
poly/.style={draw=black,thick,fill=purple,fill opacity=0.2},
bravi/.style={draw=black,thick,fill=red,fill opacity=0.4},
baseline=(center)]

\def\wun{0.63}
\def\wdeux{49/200}
\def\wtrois{0.12}
\def\wquatre{0.005}
\def\m{{2*min(\wun-\wtrois,\wdeux-\wquatre)}}
\def\undeux{\wun+\wdeux-\wtrois-\wquatre}
\def\untrois{\wun-\wdeux+\wtrois-\wquatre}
\def\a{{\undeux-2*min(\wun-\wtrois,\wdeux-\wquatre)}}

\coordinate (1312) at (\untrois,\undeux) {};
\coordinate (1213) at (\undeux,\untrois) {};
\coordinate (mx) at (\m,0) {};
\coordinate (my) at (0,\m) {};
\coordinate (ma) at (\undeux,\a) {};
\coordinate (mb) at (\a,\undeux) {};
\coordinate (center) at (0.5,0.5) {};

\node[inner sep=0.5pt,circle,draw=black,fill=black,label=below left:{$(0,0)$}] at (0,0) {};
\draw[->] (0,0) -- (1.25,0);
\draw[->] (0,0) -- (0,1.25);
\draw[axis] (1,0) -- (1,1) -- (0,1) {};
\node[vertex_carre,label=above right:{$(1234,1234)$}] at (1,1) {};

\filldraw[poly] (0,0) -- (\undeux,0) -- (1213) -- (1312) -- (0,\undeux) -- cycle;

\filldraw[bravi] (0,0) -- (mx) -- (ma) -- (1213) --(1312) -- (mb) -- (my) -- (0,0) -- cycle;

\node[vertex_carre,fill=blue] at (mx) {};
\node[vertex_carre,fill=blue] at (my) {};
\node[vertex] at (1312) {};
\node[vertex] at (1213) {};
\node[vertex] at (ma) {};
\node[vertex] at (mb) {};

\node[label=below:{$(12,0)$}] at (\undeux,0) {};
\node[label=left:{$(0,12)$}] at (0,\undeux) {};
\node[label=below:{$(\mu,0)$}] at (\m,0) {};
\node[label=left:{$(0,\mu)$}] at (0,\m) {};
\node[label=above:{$(13,12)$}] at (\untrois,\undeux) {};
\node[label=right:{$(12,13)$}] at (\undeux,\untrois) {};
\node[label=below:{$(1234,0)$}] at (1.1,0) {};
\node[label=left:{$(0,1234)$}] at (0,1) {};

\coordinate (1412) at (\wun-\wdeux-\wtrois+\wquatre,\wun+\wdeux-\wtrois-\wquatre) {};
\coordinate (1214) at (\wun+\wdeux-\wtrois-\wquatre,\wun-\wdeux-\wtrois+\wquatre) {};
\coordinate (1413) at (\wun-\wdeux-\wtrois+\wquatre,\wun-\wdeux+\wtrois-\wquatre) {};
\coordinate (1314) at (\wun-\wdeux+\wtrois-\wquatre,\wun-\wdeux-\wtrois+\wquatre) {};
\node[vertex] at (1412) {};
\node[vertex] at (1214) {};
\node[vertex] at (1413) {};
\node[vertex] at (1314) {};

\end{tikzpicture}
\caption{The 6 points illustrated from Figure \ref{fig:2qubits} in the positive orthant. We highlight that the two extra inequalities attain equality on some of these points.}
\label{fig:2qubits_bravi}
\end{figure}
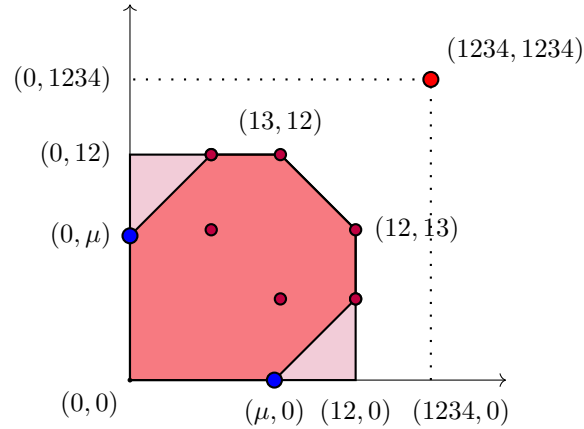

In Figure \ref{fig:2qubits}, we assumed that $w_1-w_2-w_3+w_4$ is larger than its negative $-w_1+w_2+w_3-w_4$.
This is equivalent to $ w_1 - w_3 > w_2 - w_4 $, in which case the maximum of $ x_1 - x_2 $ among the positive occupation vectors is $ (w_1+w_2-w_3-w_4) - (w_1-w_2-w_3+w_4) = 2(w_2-w_4)$. 
Otherwise, if $ w_1 - w_3 < w_2 - w_4 $ then the positive occupation vector include $ (12,23) $ and $ (13,23) $ instead of $ (12,14) $ and $ (13,14) $ in the terminology of Figure \ref{fig:2qubits}.
In this case the maximum of $ x_1 - x_2 $ among the positive occupation vectors is $ 2(w_1-w_3) $.
In either case, the maximum value of $x_1 - x_2$ over the positive occupation vectors matches the maximum over the whole polygon $ \Lambda(2,2,\bm{w}) $!
Similarly, the same holds for the minimum.

\begin{question}
Does every facet defining inequality of the polytope $ \Lambda(d,N,\bm{w}) $ achieve its maximum over positive occupation vectors of the product of $ N $ simplices $ \Delta_{d-1} $?
This would provide a closer relationship between $ \pol[L](\Delta_{d-1}^N) $ and $ \Lambda(d,N,\bm{w}) $.
\end{question}

\end{example}

\subsection{Cyclic polytopes}

Let $ S = \{a_{1}, \dots, a_{n}\} \subset \mathbb{R} $ be a fixed subset of real numbers.
The cyclic polytope $ \pol[C]_d(S) $ is the convex hull of the set $\mathbf{C}_n:=\left\{ (1, a_{i}, a_{i}^{2}, \dots, a_{i}^{d-1}) : i \in [n]\right\} \subset \mathbb{R}^{d}$.
The vertices of the cyclic polytope are naturally labelled by the set $S$.
Furthermore its face lattice and oriented matroid depend only on $n=|S|$ and not on the set $S$.
A lineup of length $ n $, i.e. a sweep, of the cyclic polytope consists of a total order of $S$ coming from a linear functional $ \bm{p} = (p_{0}, \dots, p_{d-1} ) \in \mathbb{R}^d $.
The value of the dot product in a vertex of the cyclic polytope is $(p_{0}, p_1, p_2, \dots, p_{d-1} ) \cdot (1, a_{i}, a_{i}^{2}, \dots, a_{i}^{d-1})  = p_0 + p_1a_i + \dots p_{d-1}a_i^{d-1}$.
In other words, we consider the polynomial $ p(X) = p_0 + p_1X + \dots p_{d-1}X^{d-1} \in \mathbb{R}[X] $ and order the elements of~$S$ according to the values $p(S)$.
Since the vector $ \bm{p} $ is chosen arbitrarily, finding all lineups of $\mathbf{C}_n$ is equivalent to finding all possible orderings of $n$ points induced by a \emph{polynomial} of degree at most $d-1$.

\begin{question}
Sweep polytopes depend only on the oriented matroid of the considered point configuration and the oriented matroid of $\mathbf{C}_n$ depends only on $n$.
What is the number $N(d,n)$ of sweeps of the $d$-dimensional cyclic polytope with $n$-vertices? Is there a closed formula when $ n \gg d $?
\end{question}

\newcommand{\etalchar}[1]{$^{#1}$}
\providecommand{\bysame}{\leavevmode\hbox to3em{\hrulefill}\thinspace}
\providecommand{\MR}{\relax\ifhmode\unskip\space\fi MR }
\providecommand{\MRhref}[2]{%
  \href{http://www.ams.org/mathscinet-getitem?mr=#1}{#2}
}
\providecommand{\href}[2]{#2}


\begin{thebibliography}{KBB{\etalchar{+}}08}

\bibitem[AK08]{altunbulak_pauli_2008}
Murat Altunbulak and Alexander Klyachko, \emph{The {P}auli principle
  revisited}, Commun. Math. Phys. \textbf{282} (2008), 287--322.

\bibitem[ABB{\etalchar{+}}23]{araujo_realizable_2023}
Igor Araujo, Alexander~E. Black, Amanda Burcroff, Yibo Gao, Robert~A. Krueger,
  and Alex McDonough, \emph{Realizable standard {Y}oung tableaux}, preprint,
  \href{http://arxiv.org/abs/2302.09194}{\tt arXiv:2302.09194} (February 2023),
  19~pp.

\bibitem[AB95]{avis_good_1995}
David Avis and David Bremner, \emph{How good are convex hull algorithms?},
  Proceedings of the eleventh annual symposium on Computational geometry, 1995,
  pp.~20--28.

\bibitem[AF91]{avis_pivoting_1991}
David Avis and Komei Fukuda, \emph{A pivoting algorithm for convex hulls and
  vertex enumeration of arrangements and polyhedra}, Proceedings of the seventh
  annual symposium on Computational geometry, 1991, pp.~98--104.

\bibitem[AJ18]{avis_mplrs_2018}
David Avis and Charles Jordan, \emph{mplrs: A scalable parallel vertex/facet
  enumeration code}, Math. Program. Comput. \textbf{10} (2018), no.~2,
  267--302.

\bibitem[BS00]{berenstein_coadjoint_2000}
Arkady Berenstein and Reyer Sjamaar, \emph{Coadjoint orbits, moment polytopes,
  and the {H}ilbert-{M}umford criterion}, J. Amer. Math. Soc. \textbf{13}
  (2000), no.~2, 433--466.

\bibitem[BS22]{black_flag_2022}
Alexander~E. Black and Raman Sanyal, \emph{Flag polymatroids}, preprint,
  \href{http://arxiv.org/abs/2207.12221}{\tt arXiv:2207.12221} (July 2022),
  30~pp.

\bibitem[Bra04]{bravyi_requirements_2003}
Sergey Bravyi, \emph{Requirements for compatibility between local and
  multipartite quantum states}, Quantum Inf. Comput. \textbf{4} (2004), no.~1,
  12--26.

\bibitem[BEK21]{brysiewicz_enumerating_2021}
Taylor Brysiewicz, Holger Eble, and Lukas K{\"u}hne, \emph{Enumerating chambers
  of hyperplane arrangements with symmetry}, preprint,
  \href{http://arxiv.org/abs/2105.14542}{\tt arXiv:2105.14542} (2021), 21~pp.

\bibitem[CLL{\etalchar{+}}23]{castillo_effective_2021}
Federico Castillo, Jean-Philippe Labbe, Julia Liebert, Arnau Padrol, Eva
  Philippe, and Christian Schilling, \emph{An effective solution to convex
  $1$-body {$N$}-representability}, Ann. Henri Poincaré (in print) (2023),
  81~pp.

\bibitem[CK70]{chand_algorithm_1970}
Donald~R Chand and Sham~S Kapur, \emph{An algorithm for convex polytopes},
  Journal of the ACM (JACM) \textbf{17} (1970), no.~1, 78--86.

\bibitem[CS88]{clarkson_algorithms_1998}
Kenneth~L Clarkson and Peter~W Shor, \emph{Algorithms for diametral pairs and
  convex hulls that are optimal, randomized, and incremental}, Proceedings of
  the fourth annual symposium on Computational geometry, 1988, pp.~12--17.

\bibitem[DHP22]{deza_sizing_2022}
Antoine Deza, Mingfei Hao, and Lionel Pournin, \emph{Sizing the white whale},
  preprint, \href{http://arxiv.org/abs/2205.13309}{\tt arXiv:2205.13309}
  (2022), 26~pp.

\bibitem[GNW82]{greene_probabilistic_1982}
Curtis Greene, Albert Nijenhuis, and Herbert~S Wilf, \emph{A probabilistic
  proof of a formula for the number of {Y}oung tableaux of a given shape},
  Young Tableaux in Combinatorics, Invariant Theory, and Algebra, Elsevier,
  1982, pp.~17--22.

\bibitem[KBB{\etalchar{+}}08]{khachiyan_generating_2008}
Leonid Khachiyan, Endre Boros, Konrad Borys, Khaled Elbassioni, and Vladimir
  Gurvich, \emph{Generating all vertices of a polyhedron is hard}, Discrete
  Comput. Geom. \textbf{39} (2008), no.~1-3, 174--190.

\bibitem[Kir84]{kirwan_convexity_1984}
Frances Kirwan, \emph{Convexity properties of the moment mapping. iii.},
  Invent. Math. \textbf{77} (1984), 547--552.

\bibitem[Kly04]{klyachko_quantum_2004}
Alexander Klyachko, \emph{Qmp and representations of the symmetric group},
  preprint, \href{https://arxiv.org/abs/quant-ph/0409113}{\tt
  arXiv:quant-ph/0409113} (September 2004), 47~pp.

\bibitem[Kly06]{klyachko_quantum_2006}
Alexander~A Klyachko, \emph{Quantum marginal problem and n-representability},
  Journal of Physics: Conference Series \textbf{36} (2006), no.~1, 72.

\bibitem[K{\"u}h23]{kuhne_universality_2020}
Lukas K{\"u}hne, \emph{The universality of the resonance arrangement and its
  betti numbers}, Combinatorica (2023), in press.

\bibitem[Mal07]{mallows_deconvolution_2007}
Colin Mallows, \emph{Deconvolution by simulation}, Complex datasets and inverse
  problems, IMS Lecture Notes Monogr. Ser., vol.~54, Inst. Math. Statist.,
  Beachwood, OH, 2007, pp.~1--11.

\bibitem[MV15]{mallows_young_2015}
Colin Mallows and Robert~J Vanderbei, \emph{Which {Y}oung tableaux can
  represent an outer sum?}, Journal of Integer Sequences (2015), 15.9.1, 8
  pages.

\bibitem[{OEI}23]{oeis}
{OEIS Foundation Inc.}, \emph{The {O}n-{L}ine {E}ncyclopedia of {I}nteger
  {S}equences}, 2023, Published electronically at \url{http://oeis.org}.

\bibitem[PP22]{padrol_sweeps_2021}
Arnau Padrol and Eva Philippe, \emph{Sweeps, polytopes, oriented matroids, and
  allowable graphs of permutations}, preprint,
  \href{http://arxiv.org/abs/2102.06134}{\tt arXiv:2102.06134} (December 2022),
  48~pp.

\bibitem[Pak22]{pak_what_2022}
Igor Pak, \emph{What is a combinatorial interpretation?}, preprint,
  \href{https://arxiv.org/abs/2209.06142}{\tt arXiv:2209.06142} (September
  2022), 58~pp.

\bibitem[Pos09]{postnikov_permutohedra_2009}
Alexander Postnikov, \emph{Permutohedra, associahedra, and beyond}, Int. Math.
  Res. Not. \textbf{2009} (2009), no.~6, 1026--1106.

\bibitem[Res10]{ressayre_geometric_2010}
Nicolas Ressayre, \emph{Geometric invariant theory and the generalized
  eigenvalue problem}, Invent. Math. \textbf{180} (2010), no.~2, 389--441.

\bibitem[Rot92]{rote_degenerate_1992}
G{\"u}nter Rote, \emph{Degenerate convex hulls in high dimensions without extra
  storage}, Proceedings of the eighth annual symposium on Computational
  geometry, 1992, pp.~26--32.

\bibitem[Sch15]{schilling_quantum_2015}
Christian Schilling, \emph{The quantum marginal problem}, Mathematical Results
  in Quantum Mechanics: Proceedings of the QMath12 Conference, World
  Scientific, 2015, pp.~165--176.

\bibitem[Sch98]{schrijver_theory_1998}
Alexander Schrijver, \emph{Theory of linear and integer programming}, John
  Wiley \& Sons, 1998.

\bibitem[Sta80]{stanley_weyl_1980}
Richard~P. Stanley, \emph{Weyl groups, the hard {L}efschetz theorem, and the
  {S}perner property}, SIAM J. Algebraic Discrete Methods \textbf{1} (1980),
  no.~2, 168--184.

\bibitem[Sta12]{stanley_enumerative_2012}
\bysame, \emph{Enumerative combinatorics. {V}olume 1}, second ed., Cambridge
  Studies in Advanced Mathematics, vol.~49, Cambridge University Press,
  Cambridge, 2012.

\bibitem[TOG17]{toth_handbook_2017}
Csaba~D Toth, Joseph O'Rourke, and Jacob~E Goodman, \emph{Handbook of discrete
  and computational geometry}, CRC press, 2017.

\bibitem[VW17]{vergne_inequalities_2017}
Mich\`ele Vergne and Michael Walter, \emph{Inequalities for moment cones of
  finite-dimensional representations}, J. Symplectic Geom. \textbf{15} (2017),
  no.~4, 1209--1250.

\end{thebibliography}
\end{document}